\documentclass{article}
\usepackage{amsmath,bm}   
\usepackage{amssymb}
\usepackage{graphicx}   
\usepackage{float}
\usepackage{mathrsfs}   
\usepackage{mathabx}
\usepackage{amsthm}
\usepackage{geometry}   
\usepackage{enumitem}   
\usepackage{asymptote}  
\usepackage{stmaryrd}   
\usepackage{environ}
\usepackage{chngcntr}
\usepackage{euscript}
\usepackage{galois}   
\usepackage{tikz-cd}
\usepackage{hyperref}
\usepackage[xindy, splitindex]{imakeidx}
\usepackage{xcolor}
\usepackage{tkz-euclide}
\usepackage[nottoc]{tocbibind}
\usetikzlibrary{arrows,calc,intersections}

\usepackage[all]{xy}

\newcommand{\R}{\ensuremath{\mathbb{R}}}     
\newcommand{\C}{\ensuremath{\mathbb{C}}}     
\newcommand{\Z}{\ensuremath{\mathbb{Z}}}     
\newcommand{\F}{\ensuremath{\mathbb{F}}}
\renewcommand{\P}{\ensuremath{\mathbb{P}}}

\DeclareMathOperator*{\Coker}{Coker}

\makeatletter   
\renewenvironment{proof}[1][\proofname]{\par
	\pushQED{\qed}
	\normalfont \topsep6\p@\@plus6\p@\relax
	\trivlist
	\item[\hskip\labelsep
	#1]\ignorespaces
}
{
	\popQED\endtrivlist\@endpefalse%
}
\makeatother

\NewEnviron{sol}{\begin{proof}[{\bf 解 \hspace{1em}}]\BODY\end{proof}}{}
\NewEnviron{bproof}{\begin{proof}[{\bf 略证 \hspace{1em}}]\BODY\end{proof}}{}
\renewcommand\proofname{\textbf{Proof} \hspace{1em}}

\theoremstyle{plain}
\newtheorem{theorem}{Theorem}[section]
\newtheorem{lemma}[theorem]{Lemma}

\newtheorem{proposition}[theorem]{Proposition}

\newtheorem*{remark}{Remark}

\newcommand{\bpar}{\ensuremath{\bar{\partial}}}
\newcommand{\OO}{\ensuremath{\mathcal{O}}}
\newcommand{\MM}{\ensuremath{\mathcal{M}}}
\newcommand{\NN}{\ensuremath{\mathcal{N}}}
\newcommand{\JJ}{\ensuremath{\mathcal{J}}}

\usepackage{caption}
\usepackage{comment}
\usepackage{xcolor}
\usepackage[normalem]{ulem}
\usepackage{setspace}
\title{On SYZ mirrors of Hirzebruch surfaces}
\date{}
\author{Honghao Jing}

\begin{document}
	
    \maketitle

    \begin{abstract}
        The Strominger--Yau--Zaslow (SYZ) approach to mirror symmetry constructs a mirror space and a superpotential from the data of a Lagrangian torus fibration on a K\"ahler manifold with effective first Chern class. For K\"ahler manifolds whose first Chern class is not nef, the SYZ construction is further complicated by the presence of additional holomorphic discs with non-positive Maslov index.
        
        In this paper, we study SYZ mirror symmetry for two of the simplest non-Fano toric examples: the Hirzebruch surfaces \(\F_3\) and \(\F_4\).  Our approach is to regularize moduli spaces of stable holomorphic discs using obstruction sections arising from infinitesimal deformations of the complex structure. For \(\F_3\), we determine the SYZ mirror associated to generic regularizing perturbations of the complex structure, and demonstrate that the mirror depends on the choice of perturbation. For \(\F_4\), we determine the SYZ mirror for a specific regularizing perturbation, where the mirror superpotential is an explicit infinite Laurent series. Finally, we relate this superpotential to those arising from other perturbations of $\F_4$, as determined in the literature \cite{CPS24,BGL25}, via a scattering diagram. 
    \end{abstract}

    \tableofcontents

    \section{Introduction}

        \subsection{Background and Main Results}

        Mirror symmetry, originally discovered in string theory, reveals a deep duality between complex geometry and symplectic geometry on `mirror pairs' of spaces.

        A key development in this area is the Strominger--Yau--Zaslow (SYZ) conjecture (\cite{SYZ96}), which provides a geometric interpretation of mirror symmetry. It proposes that mirror Calabi-Yau manifolds should admit dual special Lagrangian torus fibrations over a common base. 

        The SYZ conjecture extends beyond Calabi-Yau manifolds. The mirror of a non-Calabi-Yau K\"ahler manifold $X$ relative to an anticanonical divisor $D$ is known to be a Landau-Ginzburg model, i.e., a K\"ahler manifold $Y$ constructed from a moduli space of special Lagrangian tori in $X\setminus D$, together with a holomorphic function $W$ on it (\cite{HV00}). The function $W$, which is called the superpotential, encodes information about the $\mathfrak{m}_0$ obstruction in Floer homology, and is given by a sum of weighted counts of Maslov index $2$ holomorphic discs with boundary on a fixed Lagrangian torus.
        ($Y$ is actually constructed as an analytic space over a Novikov field, to avoid possible convergence issues in the definition of $W$.)

        Moreover, if some of the Lagrangian tori bound Maslov index $0$ discs, then wall-crossing phenomena occur and the actual mirror space differs from the moduli space of Lagrangian tori by instanton corrections (\cite{DA07}). (See for example \cite{HY25} for a rigorous argument.)

        Counting these holomorphic discs of index $0$ and $2$ essentially amounts to calculating certain open Gromov-Witten invariants. This is straightforward in some cases, e.g., for Fano toric manifolds. In general, the moduli spaces of holomorphic discs need not be regular. It is therefore necessary to introduce appropriate perturbations to achieve transversality, which ensures that the resulting disc counts are well defined. This phenomenon can already occur in the toric case when $X$ is not Fano. In this context, Hirzebruch surfaces, despite their relatively simple geometric structure, provide a natural testing ground for studying the SYZ mirror construction outside the Fano setting and for examining its dependence on the choice of the regularizing perturbation.

        In this paper, we regularize the disc moduli spaces on Hirzebruch surfaces by perturbing the complex structure through explicit integrable deformations, rather than perturbing the almost complex structure (\cite{J-curve}) or invoking virtual techniques (\cite{FOOObook}); this choice endows the perturbation with a more concrete geometric meaning and makes the dependence of the mirror on the choice of perturbation more transparent.

        \begin{remark}
        Before giving more specific background on Hirzebruch surfaces, we first clarify what we mean by the SYZ mirror of a toric manifold $X$ relative to the toric anticanonical divisor $D$.

As a space, the mirror is the moduli space of weakly unobstructed objects in the Fukaya category of \(X\) supported on product Lagrangian tori. Concretely, a point of this space corresponds to the following data: a product torus in
\[
    X \setminus D \cong (\mathbb{C}^*)^n,
\]
together with a unitary rank-\(1\) Novikov local system, Floer perturbation data used to regularize the moduli spaces of stable holomorphic discs in \(X\), and optionally a weak bounding cochain (which in our setting can be absorbed into the choice of local system).

This moduli space carries a natural regular function—the superpotential—defined by the Floer-theoretic obstruction term
\(\mathfrak{m}_0\), which is given by a regularized weighted count of Maslov index $2$ discs with boundary on the given product torus.

The mirror admits coordinate charts modeled on domains in \((\Lambda^*)^n\), in which the superpotential is expressed as a Laurent series (a polynomial in the Fano toric case). Distinct charts arise from different enumerative behaviors of these tori and are related by wall-crossing coordinate transformations. For Fano toric manifolds, there is only a single chart.
\end{remark}

        The $k^{\text{th}}$ \textbf{Hirzebruch surface} $\F_k$ is the total space of the projective bundle \[\P(\OO_{\P^1}\oplus\OO_{\P^1}(k)).\]Hirzebruch surfaces are basic examples of toric Kähler surfaces, whose moment polytopes can be depicted as trapezoids. See Figure \ref{image_1}, where the slanted edge of the trapezoid has slope $-\frac{1}{k}$.

        \begin{figure}[htbp] 
            \centering  
            \includegraphics[width=0.35\textwidth]{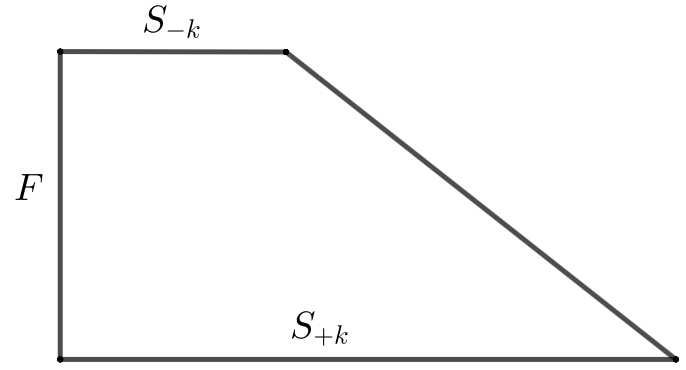}
            \caption{The $k^{\text{th}}$ Hirzebruch surface $\F_k$}  
            \label{image_1}  
        \end{figure}

        For $k=0,1$, the surfaces $\F_0,\F_1$ are Fano. The mirror of a Fano toric manifold is the Landau-Ginzburg model $((\Lambda^*)^n,W_0)$, where $\Lambda$ denotes the Novikov field
        \[\Lambda=\{\sum_{i=0}^\infty c_i T^{\lambda_i}|\ c_i\in\C,\ \lambda_i\in\R,\ \lambda_i\to\infty\}.\]The superpotential $W_0:(\Lambda^*)^n\to\Lambda$ is an analytic function (in fact a Laurent polynomial) on the mirror space $(\Lambda^*)^n$, combinatorially determined by the moment polytope of the toric manifold. More specifically, $W_0$ consists of terms associated to Maslov index $2$ holomorphic discs (with boundary on a fixed Lagrangian torus, which is represented by a point in the moment polytope) representing `basic classes' in this toric manifold. See \cite{CO06}. 

        In the case of the Hirzebruch surface $\F_k$, after suitably choosing coordinates $(x,y)$ for $(\Lambda^*)^2$, $W_0$ has the form\[W_0(x,y)=x+y+\frac{T^{\omega(S_{+k})}}{xy^k}+\frac{T^{\omega(F)}}{y}.\]Here $[F]\in H_2(\F_k)$ is the class of the fiber, $[S_{+k}]$ is the class of a section with self-intersection $k$, and $\omega(-)$ indicates the symplectic area of that class.

        For $k=2$, the surface $\F_2$ is semi-Fano. The mirror space of a semi-Fano toric surface is still $(\Lambda^*)^2$, but the superpotential $W$ differs from $W_0$ by additional terms which count virtual contribution of Maslov index $2$ stable holomorphic discs, i.e., the connected union of a disc in some basic class and a collection of spheres of zero Chern number. See \cite{CL14}.

        As in \cite[\S 3.2]{DA09} and \cite{FOOO12}, the superpotential for $\F_2$ is 
        \[W(x,y)=x+y+\frac{T^{\omega(S_{+2})}}{xy^2}+\frac{T^{\omega(F)}}{y}+\frac{T^{\omega(S_{-2}+F)}}{y}.\]Here $S_{-2}$ is the exceptional section of $\F_2$ with self-intersection $-2$.

        For $k\geq 3$, the surface $\F_k$ contains a holomorphic sphere of negative Chern number (the exceptional section $S_{-k}$), thereby allowing the existence of Maslov index $0$ stable discs. These stable discs will be the union of a disc of positive Maslov index and a collection of spheres, containing the sphere with negative Chern number. 
        
        Some of these index $0$ stable discs can be made regular (in the sense that the corresponding moduli space is regular) possibly after a perturbation of the complex structure of $\F_k$. In this way, we can observe the occurrence of the wall-crossing phenomenon, where the walls consist of Lagrangian tori that bound these index $0$ stable discs (\cite[\S 3]{DA07}). In this case, the mirror space will be a suitable gluing of various regions delimited by the walls, instead of $(\Lambda^*)^2$ simply as before.  This will provide examples (for instance Theorem \ref{thm 1.1} below) in which the walls arise from holomorphic spheres with negative Chern number, rather than from singular Lagrangian torus fibers.

        In general, the appearance of the walls depends on the chosen perturbation of the complex structure, and so do the resulting mirror space and superpotential. This dependence ultimately reflects the fact that open Gromov-Witten invariants vary with the regularizing perturbation, in contrast to their closed counterparts.
        
        In Section $2$, we illustrate this phenomenon using the example of the third Hirzebruch surface $\F_3$. We prove the following. (Let $A = \omega(S_{-3}+F)$ and $B=\omega(F)$.)

        \begin{theorem}\label{thm 1.1}
            With a generic perturbation of the complex structure on $\F_3$, the SYZ mirror of $\F_3$ is the Landau-Ginzburg model $(\F_3^\vee,W)$, where the mirror space
            \[\F_3^\vee=\{(u,v,w)\in\Lambda^2\times\Lambda^*|\ uv=1+T^Aw\}.\]The superpotential $W:\F_3^\vee\to\Lambda$ restricts to \[x+y+\frac{T^{A+2B}}{xy^3}+\frac{T^{B}}{y}+2\frac{T^{A+B}}{y^2}+\frac{T^{A}x}{y}\] on the $(x,y)$-coordinate chart, where $(x,y)=(v^{-1},w^{-1})\in(\Lambda^*)^2$. Analogously, $W$ restricts to
            \[x'+y'+\frac{T^{A+2B}}{x'y'^3}+\frac{T^{B}}{y'}+2\frac{T^{A+B}}{y'^2}+\frac{T^{2A+2B}}{x'y'^4}\]on the $(x',y')$-coordinate chart, where $(x',y')=(u,w^{-1})\in(\Lambda^*)^2$. 
            
            Each of the two affine charts corresponds to one of the two regions delimited by a wall, along which the tori bound Maslov index $0$ holomorphic discs. The position of this wall can be determined explicitly and depends on the choice of the perturbation.
        \end{theorem}

        \begin{remark}
            The mirror space $(\Lambda^*)^2$ obtained through toric perturbations (as in \cite[\S 3]{DA09}) can be viewed as an affine open subspace—the $(x,y)$-chart of $\F_3^\vee$, and the superpotential on $(\Lambda^*)^2$ is the restriction of the superpotential on $\F_3^\vee$. 
            This is in line with the principle that the toric construction of the mirror yields only a single coordinate chart of a larger mirror space.
            
            In fact, $\F_3$ admits no $T^2$-equivariant deformations of its complex structure. The regularizing deformation of the complex structure considered in \cite[\S 3]{DA09} is only $S^1$-equivariant; however, in view of its compatibility with the toric structure and the fact that the resulting mirror is $(\Lambda^*)^2$, we will still refer to such a perturbation as toric.
        \end{remark}

        In Subsection $2.2$, we consider the obstruction bundle of an irregular moduli space of index $0$ stable discs to detect regular index $0$ discs and the appearance of the wall after a generic perturbation. In Subsection $2.3$, we derive the expression of the superpotential on both sides of the wall and establish the wall-crossing transformation which implies the expression of the corrected mirror space.

        In Section $3$, we consider the fourth Hirzebruch surface $\F_4$, and fix a specific toric perturbation of $\F_4$. We prove that there are no index $0$ discs and hence no walls, and derive the expression of superpotential with respect to the fixed perturbation. The superpotential is a Laurent series rather than a Laurent polynomial, i.e., there are infinitely many index $2$ classes that contribute.

        More precisely, we prove the following. (Let $A = \omega(S_{-4}+2F)$ and $B=\omega(F)$.)

        \begin{theorem}\label{main theorem 2}
            There exists a perturbation of the complex structure on $\F_4$ for which the superpotential is a Laurent series on the mirror space $(\Lambda^*)^2$, given by
            \[W(x,y)=y+\frac{T^A}{y}+\frac{T^B}{y}\left(\sum_{k=0}^\infty (2k+1)\left(\frac{T^A}{y^2}\right)^k\right)
                +\left(x+\frac{T^{A+2B}}{xy^4}\right)\left(\sum_{k=0}^\infty (k+1)\left(\frac{T^A}{y^2}\right)^k\right).
                \]
        \end{theorem}

        At the beginning of Section $3$, we  define this specific perturbation of $\F_4$ explicitly (such that the perturbed complex structure is isomorphic to that of $\F_0$). In Subsection $3.2$, we relate the tori in $\F_4$ to a new family of Lagrangian tori in $\F_0\cong \P^1\times\P^1$ given by equations 
        \begin{equation*}
                \left\{
            \begin{array}{ll}
                |z_1^2z_2-\epsilon|=r,\\
                \mu_{S^1}(z_1,z_2)=\lambda/2, \\
            \end{array}
            \right.
        \end{equation*}
        where $\mu_{S^1}$ is the moment map of a certain $S^1$-action. This allows us to reduce to a calculation of the superpotential for these tori in $\F_0$.

        In Section $4$, we consider other perturbations of $\F_4$, and the different SYZ mirrors they induce.
        
        In Subsection $4.1$, we construct a deformation from \( \F_4 \) to \( \F_2 \). (Note that the deformation used in Section $3$ is to $\F_0$.) This deformation is also toric, so we can similarly prove that there are no index $0$ discs under it. We perform a nodal trade on the tori in \( \F_4 \), showing that they are related to the standard tori in \( \F_2 \) via an explicit wall-crossing transformation, thereby obtaining the expression of the superpotential (Proposition \ref{Vianna superpotential}). This expression, which I learned from Vianna, is also present in \cite{BGL25}. The chamber yielding this superpotential was first studied in \cite{CPS24}. 
        
        In Subsection $4.2$, we explain how the two different superpotentials for $\F_4$ we obtained, along with superpotentials for $\F_0$ and $\F_2$, are connected by an infinite sequence of wall-crossing transformations in a scattering diagram. Put differently, the chamber corresponding to Theorem \ref{main theorem 2} lies infinitely far away from the chamber corresponding to Proposition \ref{Vianna superpotential}.
        The scattering diagram is obtained through two nodal trades on the standard torus fibration of $\F_4$. 

        The significant difference between the superpotentials obtained in Theorem \ref{main theorem 2} and Proposition \ref{Vianna superpotential} not only again illustrates, in the non-Fano case, the heavy dependence of the SYZ mirror on the choice of perturbation, but also suggests that a general choice of regularization of moduli spaces from the perspective of complex algebraic geometry can lead to very different results from those arising naturally via tropical methods (as used in \cite{CPS24}, \cite{BGL25}).

        \subsection{Related Work}

        In \cite{DA23}, Auroux demonstrated that SYZ mirror constructions in the non-Fano setting may receive more corrections than previously considered, arising from stable holomorphic discs of non-positive Maslov index produced by sphere bubbling with negative Chern number. The present paper is motivated by this perspective. However, since we work in dimension two, no corrections from negative Maslov index discs appear: for any two-dimensional family of Lagrangian torus fibers, the expected dimension of the moduli space of holomorphic discs of negative index is negative, and hence no regular such moduli spaces exist. In contrast, in the setting of \cite[Theorem 1.1]{DA23}, stable holomorphic discs of non-positive Maslov index are already regular. This difference explains why, in our two-dimensional setting, we observe a new phenomenon that the SYZ mirror depends on the chosen regularizing perturbation.

        Already in \cite{FOOO10}, Fukaya, Oh, Ohta, and Ono suggested that, in the non-Fano setting, the mirror superpotential may take the form of an infinite Laurent series; see \cite[Theorem 4.6]{FOOO10}. Our Theorem \ref{main theorem 2} provides what appears to be the first explicit example of such a superpotential, arising from the fourth Hirzebruch surface. However, in \cite[Lemma 11.7]{FOOO10} the authors employ a $T^n$-equivariant virtual perturbation on moduli spaces of discs with at least one boundary marked point. Under such equivariant perturbations, Maslov $0$ discs do not occur, and the counts of Maslov $2$ discs are independent of the choice of equivariant perturbation.
        
        In contrast, in the present paper we take perturbations arising from (directions of) infinitesimal deformations of the complex structure; these yield compatible perturbations on both disc and sphere moduli spaces, and do not depend on the positions of marked points, unlike the perturbations considered in \cite{FOOO10}. Our perturbations are never $T^n$-equivariant in the sense of \cite{FOOO10}. As a consequence, the resulting disc counts—and hence the mirror—exhibit wall-crossing phenomena that depend on the choice of perturbation.

        A major recent development in mirror symmetry is the intrinsic mirror symmetry program, developed by Gross, Hacking, and Keel \cite{GHK15} and by Gross and Siebert \cite{GS19}, \cite{GS22} and others. As emphasized in \cite{GS22}, the intrinsic mirror is expected to depend only on the divisor complement $X\setminus D$, and from our perspective it is the mirror to $X\setminus D$, rather than to $X$.   However, in the non-Fano toric setting, some essential information needed to recover the corrected mirror to $X$—namely the moduli space of weakly unobstructed objects of the Fukaya category of $X$ supported on product tori in $X\setminus D$—is encoded precisely in the divisor $D$ rather than in the open part $X\setminus D\cong(\C^*)^n $. 
        Once the divisor is included, the scattering diagram associated to $(X,D)$ contains additional walls beyond those arising from $X\setminus D$ alone (Theorem \ref{thm 1.1} and \cite[Theorem 1.1]{DA23}). See the remark at the end of Section $1.1$ of \cite{DA23}. For this reason, we do not expect the intrinsic mirror approach to apply directly to the situation considered here. 

        It is also worth emphasizing that throughout this paper we consider the mirror of the pair $(X,D)$ where $X$ is a toric variety (in our case, a Hirzebruch surface) and $D$ is the union of its toric divisors ($4$ copies of $\P^1$). Although the complex structure of $X$ is deformed (and consequently the divisor $D$ may be partially smoothed) when we regularize the relevant open Gromov-Witten moduli spaces, this deformation is used solely as a perturbation to achieve transversality and determine the Floer-theoretic obstruction for product tori in $X$. We do not attempt to construct the mirror of the deformed pair $(X',D')$. 
        
        In the recent work of Berglund, Gr\"afnitz, and Lathwood \cite{BGL25}, the authors compute the superpotential of $\F_4$ and obtain the expression \eqref{superpotential2}	following the approach of Carl, Pumperla, and Siebert \cite{CPS24}. Their method relates the mirror superpotential of a non-Fano surface $X$ to that of a Fano (or semi-Fano) surface $X'$ by working with the scattering diagram associated to the deformed pair $(X',D')$. In this way, the issue of corrections arising from spheres of negative Chern number is bypassed. This approach, however, does not apply to the situation studied in Section 3 of the present paper. Under the deformation considered there, the divisor $D$ does not survive—$(X,D)$ does not deform to a pair $(X',D')$—and the superpotential \eqref{superpotential1} does not lie in any chamber of the scattering diagram used in \cite{BGL25}. A detailed discussion of the relationship between \eqref{superpotential1} and \eqref{superpotential2} is given in Subsection 4.2.

        Another relevant development is the work \cite{BCHL25} of Bardwell-Evans, Cheung, Hong, and Lin. Using symplectic methods, they construct scattering diagrams on log Calabi-Yau surfaces and recover the diagram of \cite{GHK15}. In the same way as in \cite{GHK15}, however, their scattering diagram is essentially associated to the divisor complement $X\setminus D$; in the non-Fano case, this differs from the scattering diagram associated to the pair $(X,D)$.

    \section*{Acknowledgements}
        I am deeply grateful to my advisor Denis Auroux, for suggesting this problem, for his patient guidance and constant support, as well as for his important ideas and contributions throughout this work. I would also like to thank Renato Vianna for sharing with me his calculation on a superpotential for $\F_4$ and his inspiring ideas, which facilitated the emergence of Section $4$. I would also like to thank Weixiao Lu and Zeyu Wang for helpful discussions on related algebraic geometry. This work was partially supported by NSF grant DMS-2202984.

    \section{SYZ Mirrors of the Third Hirzebruch Surface $\F_3$}

        In this section, we apply the theory of obstruction bundles to investigate the possible appearance of Maslov index $0$ discs and wall-crossing phenomena, and see their dependence on the choice of perturbation of the complex structure. We then determine the superpotential and derive the expression of the mirror space using the variable substitution formula for the superpotential.

        To be consistent in notation, we assign a coordinate $z\in \P^1$ to the exceptional section $S_{-3}$, such that in the moment polytope of $\F_3$ (see Figure \ref{image_2}, where the slanted edge should have slope $-\frac{1}{3}$), $z=0$ corresponds to the leftmost point on the line segment that represents $S_{-3}$, and $z=\infty$ corresponds to the rightmost point. Denote by $F_0$ the fiber sphere over $z=0$, and by $F_\infty$ the fiber sphere over $z=\infty$. Denote by $\sigma,\phi\in \pi_2(\F_3)$ the homotopy classes of the exceptional section $S_{-3}$ and the fiber respectively.

        \begin{figure}[htbp]  
            \centering  
            \includegraphics[width=0.35\textwidth]{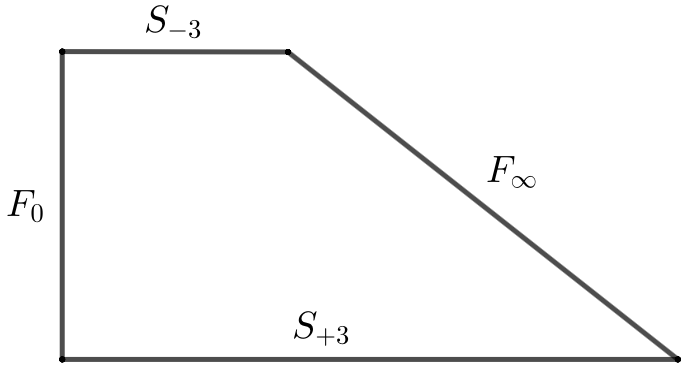}  
            \caption{The $3^{\text{rd}}$ Hirzebruch surface $\F_3$}  
            \label{image_2}  
        \end{figure}

        \subsection{Obstruction Bundle}

        Recall that if we choose the deformation as in \cite[\S 3]{DA09}, the union of the spheres $S_{-3}$ and $F_0$ deforms into a single holomorphic sphere in the class $\sigma+\phi$. Below, we will show by the theory of obstruction sheaf that, if we choose a different deformation, the union of $S_{-3}$ and another fiber will deform into a holomorphic sphere in the class $\sigma+\phi$, and this new sphere is regular.
        
        To begin with, we consider the naive moduli space $\MM(\F_3;\sigma+\phi)$ consisting of nodal spheres $S_{-3}\cup F_z$ where $z\in S_{-3}$ is the node. We identify $\MM(\F_3;\sigma+\phi)$ with $S_{-3}\cong\P^1$. This moduli space is not regular in the sense that the Dolbeault operator $\bpar$ on the normal sheaf of $S_{-3}\cup F_z$ is not surjective (see \cite[\S 2.3]{DA23}).

        \begin{remark}
            As indicated in the notation $\MM(\F_3;\sigma+\phi)$, nodal spheres of the form $S_{-3}\cup F_z$ give all the (possibly nodal) holomorphic spheres in the class $\sigma+\phi$. We will not use this.
        \end{remark}

        Denote this normal sheaf on $S_{-3}\cup F_z$ by $\NN_z$. The deformations and obstructions of $S_{-3}\cup F_z$ as a stable map are governed by $H^0$ and $H^1$ of $\NN_z$, i.e., the kernel and the cokernel of $\bpar$. These cokernels fit together to form a vector bundle over $\MM(\F_3;\sigma+\phi)$, which we call the \textbf{obstruction bundle}, and denote by $\OO b$. The Euler class of the obstruction bundle is often referred to as the \textbf{virtual fundamental class}.

        To define virtual fundamental classes in full generality, one needs to employ the machinery of Kuranishi structures (or other sophisticated methods), of which the obstruction bundles appearing here are special cases. For this reason we will not present the full definition in detail. 

        In the following, we denote $S_{-3}$ by $S$ for simplicity.

        \begin{proposition}\label{Ob for F_3}
            The obstruction bundle $\OO b$ over $\MM(\F_3;\sigma+\phi)\cong\P^1$ is isomorphic to $\OO_{\P^1}(1)$.
        \end{proposition}

        \begin{proof}
            We first determine the normal sheaf of each nodal sphere.
            
            \textbf{Claim.} The normal sheaf $\NN_z$ of $S\cup F_z$ in $\F_3$ is a line bundle that restricts to $\OO_{\P^1}(-2)$ and $\OO_{\P^1}(1)$ on $S$ and $F_z$ respectively.

            In fact, the normal sheaf of any effective Cartier divisor $D$ in $\F_3$ is the sheaf $\OO_{\F_3}(D)|_D$. Taking $D$ to be $S+F_z$, the restriction of the normal sheaf to $S$ is
            \begin{align*}
                \OO_{\F_3}(S+F_z)|_S&\cong\OO_{\F_3}(S)|_S\otimes\OO_S(S\cap F_z)\\
                &\cong\NN_S\otimes\OO_{\P^1}(1)\\
                &\cong\OO_{\P^1}(-2).
            \end{align*}
            Similarly, we can obtain the restriction of $\NN_z$ to $F_z$. The claim is proved.

            We next identify the cohomology group $H^1(S\cup F_z,\NN_z)$.

            \textbf{Claim.} The group $H^1(S\cup F_z,\NN_z)$ is canonically isomorphic to $H^1(S,(\OO_S(-3))(z))$.

            In fact, we have a short exact sequence of sheaves on $S\cup F_z$
            \[0\to\NN_z\to i_{*}\NN_z|_S\oplus i'_{*}\NN_z|_{F_z}\to\C_z\to 0.\]Here $i,i'$ are inclusions of two components, and $\C_z$ is the skyscraper sheaf supported on $z$. From the long exact sequence of cohomology, we can see that
            \[H^1(S\cup F_z,\NN_z)\cong H^1(S,\NN_z|_S)\cong H^1(S,\OO_S(-3)\otimes\OO_S(z)).\]The claim is then proved.

            Back to the proposition, by the claim above, the obstruction bundle $\OO b$ is (canonically isomorphic to) a line bundle on $\P^1$ with fiber $H^1(\P^1,\OO_{\P^1}(-3)(z))$ over each point $z\in\P^1$. By a family version of Serre's duality, the dual $\OO b^\vee$ is a line bundle with fiber $H^0(\P^1,\OO_{\P^1}(1)(-z))$. Therefore, \[\OO b^\vee\cong \pi_{0,*}(\pi^*_1(\OO_{\P^1}(1))(-\Delta)),\]where $\pi_0,\pi_1:\P^1\times\P^1\to\P^1$ are the projections to each factor, $\Delta\subset\P^1\times\P^1$ is the diagonal divisor. 
            
            Since $\OO_{\P^1\times\P^1}(\Delta)=\pi_0^*\OO_{\P^1}(1)\otimes\pi_1^*\OO_{\P^1}(1)$, we have $\OO b^\vee\cong\OO_{\P^1}(-1)$. The proposition follows.
        \end{proof} 

        Next, we want to deform the complex structure $J$ on $\F_3$ to regularize the moduli space $\MM(\F_3;\sigma+\phi)$. To do this, we take a transverse section of the obstruction bundle $\OO b$. Its intersection with the zero section is a submanifold of $\MM(\F_3;\sigma+\phi)$ which represents the virtual fundamental class. A key fact is that sections of this sort can be provided by deformations of $J$ (Proposition \ref{deformation gives a section} below).

        Before that, we briefly review the pseudo-holomorphic curve equation and regularity of solutions, as a preparation for the proof of Proposition \ref{deformation gives a section}.

        Consider a smooth map
        \[u:\Sigma\to X\]from a possibly nodal closed Riemann surface $(\Sigma,j)$ to a symplectic manifold with a compatible almost complex structure $(X,J)$. The \textbf{pseudo-holomorphic curve equation} is\[\bpar_{J}u=\frac{1}{2}(du+J\circ du\circ j)=0,\]where $\bpar_J(u)\in\Omega^{0,1}(\Sigma,u^*TX)$. The \textbf{linearized operator}
        \begin{equation}\label{original linearized operator}
            D_{\bpar_J,u}:\Omega^0(\Sigma,u^*TX)\to\Omega^{0,1}(\Sigma,u^*TX)
        \end{equation}
        is a real linear Cauchy-Riemann operator on bundle $u^*TX$ whose principal part is the standard Dolbeault operator $\bpar$. The kernel and cokernel of $D_{\bpar_J,u}$ gives the first-order deformation and obstruction of $u$ (i.e., the tangent space and irregularity of the moduli space). 
        
        In this way, however, when $\Sigma$ is nodal, $u$ is only allowed to deform into another nodal curve, i.e., the domain $\Sigma$ is unchanged. If deforming into a smooth curve is allowed, we need to enlarge the domain of the operator $D_{\bpar_J,u}$ from smooth vector fields along $u$ to vector fields along $u$ which can possibly contain poles at nodes. Assuming for simplicity that the components of $u(\Sigma)$ are immersed and intersect transversely at the nodes, the domain should include meromorphic sections of $u^*TX$ on each component of $\Sigma$, with at most a simple pole at each node, satisfying that the residues from different components coincide at the node.

        Then, we can quotient out the tangent component (which corresponds to reparametrizations) of the above described sections and get the sections of the normal sheaf $\NN_u$.
        In this way, for the moduli space consisting of both smooth curves and nodal curves, the operator
        \[\Tilde{D}_{\bpar_J,u}:\Omega^0(\Sigma,\NN_u)\to\Omega^{0,1}(\Sigma,\NN_u)\]governs the deformation and obstruction. We also see from the first claim in the proof of Proposition \ref{Ob for F_3} that the sections of $\NN_u$ are exactly the sections with poles at nodes whose residues match. See \cite[\S 2.3]{DA23} for a detailed discussion.

        From above, we see that the obstruction bundle is a bundle on the moduli space with fiber
        \[\OO b_u=\Coker\Tilde{D}_{\bpar_J,u}\cong H^1(\Sigma,\NN_u).\]

        Now we return to the case of $\F_3$.

        \begin{proposition}\label{deformation gives a section}
            A deformation $\{J(t),t\in(-\epsilon,\epsilon)\}$ of the complex structure $J=J(0)$ on $\F_3$ gives rise to a holomorphic section $s$ of the obstruction bundle $\OO b\to\MM(\F_3;\sigma+\phi)$.
        \end{proposition}

        \begin{proof}
            Recall that each element in $\MM(\F_3;\sigma+\phi)$ is in fact a map $u:\Sigma\to\F_3$ where $\Sigma$ is the nodal sphere $\P^1\cup\P^1$ and $u$ satisfies the pseudo-holomorphic (in fact holomorphic) curve equation.
            
            To construct a section of $\OO b$ from the deformation $\{J(t)\}$, we consider the infinitesimal deformation $\dot{J}(0)$ and the expression
            \[\dot{J}(0)
        \circ du\circ j\in\Omega^{0,1}(\Sigma,u^*T\F_3).\]
        Its image under the natural map $\Omega^{0,1}(\Sigma,u^*T\F_3)\to\Omega^{0,1}(\Sigma,\NN_u)$, which we denote by $\dot{J}(0)
        \circ du\circ j$ as well, projects to an element $\mathrm{pr}(\dot{J}(0)
        \circ du\circ j)$ in $\mathrm{Coker}\Tilde{D}_{\bpar_J,u}\cong H^1(S,\OO_S(-3)(z))$. In this way, 
        \[s:u\mapsto\mathrm{pr}(\dot{J}(0)
        \circ du\circ j)\]gives the desired section of $\OO b$.

            In the following, we show that the section $s$ is holomorphic. 
            
            Consider another obstruction bundle $\OO b'\to\MM(\F_3;\sigma+\phi)$ which is trivial of rank $2$ with fiber $H^1(S,\OO_S(-3))$. Now $\OO b'$ governs the deformation theory of the nodal spheres in $\MM(\F_3;\sigma+\phi)$ that requires the deformed ones to remain nodal. Same as before, the deformation $\{J(t)\}$ gives a section $s'$ of $\OO b'$\[s':u\mapsto\mathrm{pr}'(\dot{J}(0)\circ du\circ j)\in\Coker D_{\bpar_J,u}.\] Here $D_{\bpar_J,u}$ is the linearized operator \eqref{original linearized operator}, whose cokernel is $H^1(S,\OO_S(-3))$, i.e., the fiber of $\OO b'$. This section $s'$ has to be constant, because we have the identification \[\mathrm{pr}'(\dot{J}(0)\circ du\circ j)=\mathrm{pr}'(\dot{J}(0)\circ di_S\circ j_S).\]Here $i_S$ is the inclusion of the sphere $S$, which can be also seen as a holomorphic map, and the corresponding linearized operator has the same cokernel as that of $D_{\bpar_J,u}$. Since the right hand side does not depend on $u\in\MM(\F_3;\sigma+\phi)$, the section $s'$ of $\OO b'$ is constant.

            Finally, the proposition follows from the observation that $\OO b$ is naturally a holomorphic quotient bundle of $\OO b'$, and $s$ is the image of $s'$ under the quotient map. This observation can be summarized by the following commutative diagram.
            \[\begin{tikzcd}
	{\dot{J}(0)\circ du\circ j} & {\Omega^{0,1}(\Sigma,u^*T\F_3)} & {H^1(\Sigma,\NN_u)} & {H^1(S,\OO_S(-3)(z))} \\
	{\dot{J}(0)\circ di_S\circ j_S} & {\Omega^{0,1}(S,i_S^*T\F_3)} & {H^1(S,\NN_S)} & {H^1(S,\OO_S(-3))}
	\arrow["\in"{description}, draw=none, from=1-1, to=1-2]
	\arrow[maps to, from=1-1, to=2-1]
	\arrow["{\mathrm{pr}}", from=1-2, to=1-3]
	\arrow["{\mathrm{res}}"', from=1-2, to=2-2]
	\arrow["\cong", from=1-3, to=1-4]
	\arrow["\in"{description}, draw=none, from=2-1, to=2-2]
	\arrow["{\mathrm{pr}'}", from=2-2, to=2-3]
	\arrow["\cong", from=2-3, to=2-4]
	\arrow["{\text{(natural surjection)} }"', two heads, from=2-4, to=1-4]
\end{tikzcd}\]
        \end{proof}

        We next characterize the intersection of the above constructed section $s$ with the zero section.

        \begin{proposition}\label{zero set of section}
            The zero set $s^{-1}(0)$ consists of nodal spheres $u\in\MM(\F_3;\sigma+\phi)$ that deform into a holomorphic sphere after deforming $J$.
        \end{proposition}

        \begin{proof}
            First, we suppose that the nodal sphere $u$ deforms into a holomorphic sphere after deforming $J$. Let $V\in\Omega^0(\Sigma,\NN_u)$ be the infinitesimal deformation of $u$. By differentiating the equation $\bpar_{J}u=\frac{1}{2}(du+J\circ du\circ j=0)$, we have
            \begin{equation}\label{Coker of D}
                \Tilde{D}_{\bpar_J,u}V+\frac{1}{2}\dot{J}(0)
        \circ du\circ j=0,
            \end{equation}
            which means $s(u)=\mathrm{pr}(\dot{J}(0)
        \circ du\circ j)=0\in\Coker\Tilde{D}_{\bpar_J,u}$.

            Conversely, if $\mathrm{pr}(\dot{J}(0)\circ du\circ j)=0$, there exists $V$ such that equation \eqref{Coker of D} holds. However, $V$ is only a first-order deformation, we cannot guarantee the existence of an actual deformed sphere corresponding to $V$ at the moment. 
            
            We address this problem by considering the parametrized moduli space
            $\MM_\JJ(\F_3;\sigma+\phi)$ which consists of pairs $(u,J')$ that satisfies $\bpar_{J'}u=0$ where $u$ is a map from $\P^1$ or $\P^1\cup\P^1$ to $(\F_3,J')$, in the class $\sigma+\phi$. $J'$ is a compatible complex structure on the symplectic manifold $\F_3$ which takes values in $\JJ$, a two-parameter family of complex structures that gives a semiuniversal deformation described in \cite[\S 2.3]{MM04}. $\JJ$ is semiuniversal in the sense that every infinitesimal deformation of $\F_3$ is obtained by a pullback of the deformation $\JJ$, and the Kodaira-Spencer map associated to $\JJ$ is an isomorphism. Moreover, $\JJ$ can be identified with $\C^2$ as a complex manifold, where the origin corresponds to the complex structure $J$ of $\F_3$.

            \textbf{Claim.} The parametrized moduli space $\MM_\JJ(\F_3;\sigma+\phi)$ is regular at each point $(u,J)$ where $J$ is the complex structure of $\F_3$, and $u\in\MM(\F_3;\sigma+\phi)$.

            In fact, the corresponding linearized operator of $\MM_\JJ(\F_3;\sigma+\phi)$ at $(u,J)$ is\[(V,Y)\mapsto\Tilde{D}_{\bpar_J,u}V+\frac{1}{2}Y\circ du\circ j,\]where $V,Y$ are the infinitesimal deformation of $u,J$ respectively. This operator is surjective if and only if there exists $Y$ such that $\mathrm{pr}(Y\circ du\circ j)\neq0$. This holds indeed because of the surjectivity of the following composition of maps\[\begin{tikzcd}
	{T_0\JJ} & {H^1(\F_3,T\F_3)} & {H^1(S,i^*_ST\F_3)} & {H^1(S,\NN_S)} & {H^1(\Sigma,\NN_u),}
	\arrow["\cong", from=1-1, to=1-2]
	\arrow["\cong", from=1-2, to=1-3]
	\arrow["\cong", from=1-3, to=1-4]
	\arrow[two heads, from=1-4, to=1-5]
\end{tikzcd}\]where the first map is the Kodaira-Spencer map of $\JJ$, and the last surjection appeared in the commutative diagram in Proposition \ref{deformation gives a section}. The composition is exactly given by $Y\mapsto\mathrm{pr}(Y\circ du\circ j)$. The claim is then proved.

            Back to the proposition, if $\mathrm{pr}(\dot{J}(0)\circ du\circ j)=0$, there exists $V$ such that equation \eqref{Coker of D} holds, which means $(V,\dot{J}(0))$ belongs to the tangent space of $\MM_\JJ(\F_3;\sigma+\phi)$. Since $\MM_\JJ(\F_3;\sigma+\phi)$ is regular at $(u,J)$, there exists an actual deformation of complex structure and $u$ realizing $V$. (However, this actual deformation of $J$ does not necessarily coincide with the chosen deformation $J(t)$. The following Lemma \ref{continuing previous proposition} will address this issue and complete the proof of Proposition \ref{zero set of section}.)

        
            This deformed sphere must be smooth. If not, we consider the obstruction bundle $\OO b'\to\MM(\F_3;\sigma+\phi)$ which is trivial of with fiber $H^1(S,\OO_S(-3))$. $\OO b'$ governs the deformation theory of the nodal spheres in $\MM(\F_3;\sigma+\phi)$ that requires the deformed ones to remain nodal. Since non-trivial holomorphic sections of $\OO b'$ have no zeros, we see that the deformed sphere cannot be nodal, and is hence smooth.
        \end{proof}

        In the following Lemma \ref{continuing previous proposition}, we demonstrate further properties of the parametrized moduli space $\MM_\JJ(\F_3;\sigma+\phi)$ and complete the proof of Proposition \ref{zero set of section}. 

        We consider the projection
        \begin{align*}\pi:\MM_\JJ(\F_3;\sigma+\phi)&\to\JJ,\\
            (u',J')&\mapsto J'.
        \end{align*}
        We view the chosen deformation  $\{J(t)\}$ as a path in the family $\JJ\cong\C^2$ so that $J(0)=0$ corresponds to the complex structure $J$ of $\F_3$. In this way, $\MM(\F_3;\sigma+\phi)\cong\P^1$ can be identified with the central fiber $\pi^{-1}(0)$ in $\MM_\JJ(\F_3;\sigma+\phi)$.

        \begin{lemma}\label{continuing previous proposition}
            If $(V,\dot{J}(0))$ belongs to the tangent space $T_{(u,J)}\MM_{\JJ}(\F_3;\sigma+\phi)$, the path $\{J(t)\}$ can be lifted to a path in $\MM_\JJ(\F_3;\sigma+\phi)$ passing through $(u,J)$ in a small neighborhood of $0$.
        \end{lemma}

        \begin{proof}
            By the claim in Proposition \ref{zero set of section}, $\MM_\JJ(\F_3;\sigma+\phi)$ is regular in a neighborhood of $\pi^{-1}(0)$. Further, $\MM_\JJ(\F_3;\sigma+\phi)$ is a two-dimensional complex manifold in a neighborhood of $\pi^{-1}(0)$ on which the restriction of the projection $\pi$ is holomorphic. (The linearization of the operator that defines $\MM_\JJ(\F_3;\sigma+\phi)$ is complex linear, and the linearization of $\pi$ is complex linear as well. The expected dimension follows from the index formula.) 

            \textbf{Claim.} The normal bundle of $\pi^{-1}(0)\cong\P^1$ in $\MM_\JJ(\F_3;\sigma+\phi)$ is isomorphic to $\OO(-1)$.

            To prove the claim, we consider the following composition of maps\[\begin{tikzcd}
	{T_0 \JJ} & {H^1(\F_3,T\F_3)} & {H^1(S,\OO_S(-3))} & {\Gamma(S,\OO b)} && {S.}
	\arrow["\cong", from=1-1, to=1-2]
	\arrow["\cong", from=1-2, to=1-3]
	\arrow["\cong", from=1-3, to=1-4]
	\arrow["{s\mapsto s^{-1}(0)}", dashed, two heads, from=1-4, to=1-6]
\end{tikzcd}\]
            Here we identify $\MM(\F_3;\sigma+\phi)$ with $S$ via the nodal point of a stable disc. $\Gamma(S,\OO b)$ denotes the holomorphic sections of $\OO b$. The third map is given by identifying $H^1(S,\OO_S(-3))$ with $\Gamma(S,\OO b')$ and the surjective bundle map $\OO b'\to\OO b$. The composition $T_0\JJ\to\Gamma(S,\OO b)$ of the first three maps recovers the construction in Proposition \ref{deformation gives a section}. The last map maps a nonzero section $s$ to its zero. The full composition $T_0\JJ\setminus\{0\}\to S$ can be identified with the standard quotient map $(\C^2)^*\to\P^1$.

            Hence, we have a short exact sequence\[0\to T_u\MM(\F_3;\sigma+\phi)\to T_{(u,J)}\MM_\JJ(\F_3;\sigma+\phi)\xrightarrow{d\pi}L_u\to 0,\]
            where $d\pi$ is the differential of $\pi$, and $L_u$ is the line in $T_0\JJ$ consists of the preimages of $u$ under the map $T_0\JJ\to S$, i.e., the tautological line of $u$. The claim follows. 

            By the claim, we can blow down the exceptional divisor $\pi^{-1}(0)$ and obtain the blow-up
            \[p:\hat{U}\to U.\]Here $\hat{U}$ is a tubular neighborhood of $\pi^{-1}(0)$ in $\MM_\JJ(\F_3;\sigma+\phi)$. $U$ is an open subset in $\C^2$ which contains $0$, such that $p^{-1}(0)=\pi^{-1}(0)$. 

            Taking the inverse of $p$ on the complement of $p^{-1}(0)$, we obtain a map $\pi\circ p^{-1}:U\setminus \{0\}\to \JJ$, which extends to a map $f:U\to\JJ$ sending $0$ to $0$ by Hartogs's theorem (this is essentially proving the universal property of the  blow-down, see for example \cite[\S 2.13]{AlgebraicSurfaces}).  One checks that the differential $df$ is surjective at $0$ using the information of $d\pi$ and $dp$. Thus, $f$ is a biholomorphism in a neighborhood of $0$. We have identified the map $\pi$ locally as the blow-up of $\JJ$ at $0$.

            The lemma now follows from a path-lifting property of the blow-up of $\C^2$.
        \end{proof}

        By Lemma \ref{continuing previous proposition} above, we see that if the stable sphere $u$ admits a first-order deformations $V$ in the family of complex structures $\{J(t)\}$, then $u$ admits actual deformations in this family, since the actual deformations correspond to elements in $\MM_\JJ(\F_3;\sigma+\phi)$. Hence, the proof of Proposition \ref{zero set of section} is now complete.

        By Proposition \ref{zero set of section} and Proposition \ref{Ob for F_3}, we have shown that, by a generic deformation of complex structure, there exists a unique fiber $F_z$ such that $S\cup F_z$ deforms into a holomorphic sphere $E$ in the class $\sigma+\phi$. $E$ is regular since its normal bundle is $\OO_E(-1)$.
        
        \subsection{Maslov Index $0$ Holomorphic Discs and Walls}

        Recall that in a (special) Lagrangian torus fibration, the walls consist of Lagrangian tori that bound regular index $0$ discs. (When speaking of an index $0$ disc, we always assume it is non-constant.) In this subsection, we investigate the existence of possible index $0$ discs and walls after a perturbation (small deformation) of $J$. 

        Fix a product torus $L$ in $\F_3$ (a non-degenerate orbit of the torus action on $\F_3$).  $L$ is a Lagrangian torus.  The homotopy group $\pi_2(\F_3,L)$ is generated by classes $\beta_1,\beta_2,\sigma,\phi$. Here $\beta_1,\beta_2$ are the basic classes indicated in Figure \ref{image_3}.

        \begin{figure}[htbp]  
            \centering  
            \includegraphics[width=0.35\textwidth]{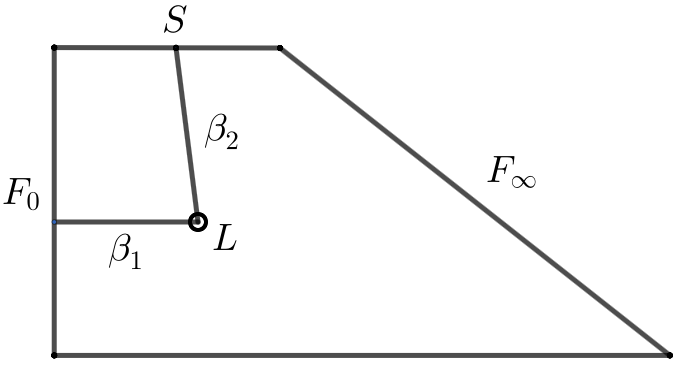}  
            \caption{Disc classes in $\pi_2(\F_3,L)$}  
            \label{image_3}  
        \end{figure}

        When $J$ is unperturbed, there are already some candidates for index $0$ discs with boundary on $L$: discs in the class $\beta_2+\sigma$, $\beta_2+3\sigma+\phi$ and so on. These are irregular stable discs, but could possibly become regular (as discs with boundary on a one-parameter family of product tori) after perturbing $J$. 
        
        We first note that among these candidates, there is one stable disc deforming into a regular holomorphic disc indeed, which we explain in the following. 

        For a fixed torus $L$, we have an $S^1$-family of standard holomorphic discs in the basic class $\beta_2$, each is a portion of a fiber sphere that intersects with $L$ in a circle, and is actually the upper part delimited by $L$ of that fiber sphere.
        
        Now, we consider a family of stable discs: each is the union of a standard disc in the class $\beta_2$ and the sphere $S=S_{-3}$. Then, we let the torus $L$ vary from $F_0$ to $F_\infty$, forming a family $\Tilde{L}$. (In the moment polytope picture, these tori form a curve from the left boundary to the right boundary.) The nodal discs described above with boundary on one of the tori in $\Tilde{L}$ form a family, parameterized by the nodal point $z\in S\setminus\{0,\infty\}$. In other words, the corresponding moduli space $\MM(\F_3,\Tilde{L};\beta_2+\sigma)$ is $\P^1\setminus\{0,\infty\}$, identified with an open subset of $\MM(\F_3;\phi+\sigma)\cong\P^1$.

        The following proposition is inspired by the work of Chan \cite{Chan11} and shows that some of the results in \cite{Chan11} extend beyond the semi-Fano setting.

        \begin{proposition}
        $\MM(\F_3,\Tilde{L};\beta_2+\sigma)$ has the same obstruction bundle as $\MM(\F_3;\phi+\sigma)$. The sections given by the perturbation of $J$ (Proposition \ref{deformation gives a section}) are the same as well.
        \end{proposition}

        \begin{proof}
            One can find the obstruction sheaf by the same steps as in Proposition \ref{Ob for F_3}. For a stable disc in $\MM(\F_3,\Tilde{L};\beta_2+\sigma)$, its normal sheaf restricts to $\OO_S(-3)(z)$ on $S$, and sections of the normal sheaf over the disc component are sections of the trivial line bundle that are allowed to have a pole at the node. The domain of the operator $\bpar$ consists of sections of the normal sheaf whose restriction on the boundary of the disc component lies in a trivial real line subbundle (given by the tangent spaces of the Lagrangian torus). 
            
            Thus, the operator $\bpar$ is surjective on the disc component. The cokernel of $\bpar$ on the stable disc is canonically isomorphic to $H^1(S,\OO_S(-3)(z))$, which coincides with the case of $\MM(\F_3;\phi+\sigma)$. Hence, their obstruction bundles coincide.

            By going over the construction in Proposition \ref{deformation gives a section}, one can similarly show that the sections given by the perturbation are the same as well.
        \end{proof}
        
        Therefore, assuming that the zero of the section $s$ of the obstruction bundle $\OO b$ is not $0$ or $\infty$, with boundary on this family $\Tilde{L}$ of product tori, there is a unique disc in the class $\beta_2$ whose union with the sphere $S$ deforms into an index $0$ disc in the class $\beta_2+\sigma$ after perturbation. This disc in the class $\beta_2$ is exactly the upper part of the sphere $F_z$ we mentioned at the end of the previous subsection (after Lemma \ref{continuing previous proposition}).

        \begin{remark}
            In \cite[\S 11]{FOOO10}, Fukaya, Oh, Ohta, and Ono show that, after choosing a perturbation equivariant with respect to the torus action, there are no index $0$ discs. This does not contradict the result here since we allow non-equivariant perturbations; in fact, perturbations arising from deformations of the complex structure are never $T^2$-equivariant. Certain $S^1$-equivariant perturbations (such as the one in \cite[\S 3]{DA09}), which we  refer to as toric perturbations, correspond to the cases where the zero of the section $s$ of the obstruction bundle is $0$ or $\infty$. 
        \end{remark}

        We next use an argument of intersection numbers to exclude the product tori away from $F_z$ from bounding index $0$ discs after perturbation.  

        \begin{proposition}\label{wall of F_3}
            If the torus $L$ is disjoint from the fiber $F_z$, it cannot bound any index $0$ discs after perturbation (a sufficiently small deformation along the chosen family $\{J(t)\}$).
        \end{proposition}

        \begin{proof}
            Assume $L$ lies on the right side of $F_z$ in the moment polytope picture. Suppose $L$ bounds an index $0$ stable disc in the class $m_1\beta_1+m_2\beta_2+m_3\sigma+m_4\phi$, and this stable disc deforms into a holomorphic disc after perturbation. By positivity of intersection, it should have non-negative intersection with the spheres $F_0,F_\infty,S_{+3},E$. By a computation of intersection numbers, we have
            \begin{align*}
                m_1+m_3&\geq0,\\
                m_3&\geq0,\\
                m_4&\geq0,\\
                m_1+m_2-2m_3+m_4&\geq0,\\
                m_1+m_2-m_3+2m_4&=0,
            \end{align*}
            where the last equality comes from the Maslov index condition.

            From these relations, the only possible classes are $m\beta_1-m\beta_2$, $m\in\Z_{>0}$.  If there exists a $J(t_n)$-holomorphic disc in the class $m\beta_1-m\beta_2$ for each $t_n$ in a sequence $\{t_n\}$ that tends to $0$, then by the Gromov compactness, these discs converge to a stable disc in the same class in $\F_3$.

            However, stable discs in the class $m\beta_1-m\beta_2$ do not exist in $\F_3$: since $m\beta_1-m\beta_2$ has negative intersection with $S$, it must contain at least one copy of $S$, which contradicts the fact that its intersection with $F_\infty$ is $0$. Hence, by taking a sufficiently small deformation along the family $\{J(t)\}$, there are no index $0$ discs.

            The proof for the case where $L$ lies on the left side of $F_z$ is similar.
        \end{proof}

        From this proposition, we can conclude that when the zero $z$ of the section $s$ of the obstruction bundle is not $0$ or $\infty$, the wall consists of Lagrangian tori that intersect the sphere $F_z$. In the moment polytope, the wall is exactly the line segment that represents $F_z$. The peculiarity of this example lies in the fact that the wall does not arise from or pass through a singular Lagrangian fiber, as in previously discovered examples.

        On the other hand, when $z$ equals $0$ or $\infty$, there are no walls. 

        To conclude this subsection, we take one step further to prove that the classes of index $0$ discs that may appear on the wall can only be multiples of $\beta_2+\sigma$ (where these discs might arise from multiple covers of the deforming disc in class $\beta_2+\sigma$).

        \begin{proposition}\label{index 0 discs for F_3}
            If z is not $0$ or $\infty$ and the torus $L$ intersects $F_z$, the only classes that may contain any index $0$ discs after perturbation are $m(\beta_2+\sigma)$ for positive integers $m$.
        \end{proposition}

        \begin{proof}
            When $L$ intersects the fiber $F_z$ ($L$ lies on the wall), it is difficult to define a reasonable intersection number of discs it bounds with the sphere $E$. Hence, we need to replace $E$ with other holomorphic spheres to apply the intersection number argument.

            \textbf{Claim.} For any $z_1,z_2\in\P^1$, the spheres $F_{z_1}$, $F_{z_2}$ and $S$ deform into a holomorphic sphere, which we denote by $E_{z_1,z_2}$.

            In fact, the stable sphere $F_{z_1}+F_{z_2}+S$ is regular—the
            $H^1$ of its normal sheaf is zero (which one can show by the same method as in Proposition \ref{Ob for F_3}). The claim follows.

            Suppose $L$ bounds an index $0$ stable disc in the class $m_1\beta_1+m_2\beta_2+m_3\sigma+m_4\phi$, and this stable disc deforms into a holomorphic disc after perturbation. It should have non-negative intersection with the spheres $F_0,F_\infty,S_{+3},E_{0,0},E_{\infty,\infty}$. We hence have
            \begin{align*}
                m_1+m_3&\geq0,\\
                m_3&\geq0,\\
                m_4&\geq0,\\
                2m_1+m_2-m_3+m_4&\geq0,\\
                m_2-m_3+m_4&\geq0,\\
                m_1+m_2-m_3+2m_4&=0.
            \end{align*}
            From these relations, the only possible classes are $m\beta_2+m\sigma$, $m\in\Z_{>0}$. 
        \end{proof}

        \subsection{Maslov Index $2$ Holomorphic Discs and Superpotential}
        
        In this subsection, we determine the superpotential on both sides of the wall. These two expressions of the superpotential are related by a change of variables. Since the superpotential is an analytic function on the mirror space, the change of variables then provides a gluing formula for affine charts of the mirror space. (See \cite[\S 3]{DA07} or \cite[\S A]{AAK16}.)

        In this subsection, we fix a generic perturbation of $J$. As before, there exists a unique fiber $F_z$ such that $S\cup F_z$ deforms into a holomorphic sphere $E$. We assume throughout the subsection that the $z$ does not equal $0$ or $\infty$.

        For a product torus $L$ in $\F_3$, we first find all possible classes in $\pi_2(\F_3,L)$ that may contribute to the superpotential. Recall that the superpotential is a weighted count of regular index $2$ discs that $L$ bounds.

        \begin{proposition}
            If the torus $L$ lies on the right side of $F_z$, the only possible classes that may contain index $2$ discs after perturbation are the basic classes $\beta_1,\beta_2,-\beta_2+\phi,-\beta_1+3\beta_2+\sigma$ and extra classes $2\beta_2+\sigma,\beta_1+\beta_2+\sigma$.
        \end{proposition}

        \begin{proof}
            Suppose $L$ bounds an index $2$ stable disc in the class $m_1\beta_1+m_2\beta_2+m_3\sigma+m_4\phi$, and this stable disc deforms into a holomorphic disc after perturbation. Similar to Proposition \ref{wall of F_3}, we have relations obtained from intersection numbers and Maslov index
            \begin{align*}
                m_1+m_3&\geq0,\\
                m_3&\geq0,\\
                m_4&\geq0,\\
                m_1+m_2-2m_3+m_4&\geq0,\\
                m_1+m_2-m_3+2m_4&=1.
            \end{align*}
            From these, we can see that the only possible classes are
            \begin{enumerate}
                \item $m_3=0,\ m_4=0:\ m\beta_1+(1-m)\beta_2$,  $\ \ m\geq0$;
                \item $m_3=0,\ m_4=1:\ m\beta_1+(-1-m)\beta_2+\phi$,  $\ \ m\geq0$;
                \item $m_3=1,\ m_4=0:\ m\beta_1+(2-m)\beta_2+\sigma$,  $\ \ m\geq-1$.
            \end{enumerate}
            
            In case $1$, the coefficient $(1-m)$ of $\beta_2$ has to be non-negative, by considering its intersection with the exceptional section $S$ as in Proposition \ref{wall of F_3}. In case $2$, we can similarly see that the coefficient $(-1-m)$ of $\beta_2$ is at least $-1$, since we now have a class $\phi$ that also contributes to the intersection with $S$.
            
            For case $3$, we can suppose the stable disc contains $k$ copies of $S$. Since its intersection number with $F_\infty$ is $1$, we have $k\leq 1$. By checking the intersection of $S$ and the stable disc minus $k\cdot S$ (which is non-negative if $k=0$ and is positive if $k=1$), we see that $-1\leq m\leq 1$.
        
            To sum up, we have $6$ possible classes, which match the classes listed in the proposition.
        \end{proof}

        Next, we need to understand the coefficients of these classes in the superpotential, which boils down to counting regular index $2$ discs whose boundary passes through a generic point of the torus $L$. However, making use of the result from \cite{DA09}, we have a more convenient solution.

        \begin{proposition}
            If the torus $L$ lies on the right side of $F_z$, the superpotential for $\F_3$ with respect to the chosen perturbation has the form
            \[W_{\mathrm{right}}(x,y)=x+y+\frac{T^{\omega(S_{-3}+3F)}}{xy^3}+\frac{T^{\omega(F)}}{y}+2\frac{T^{\omega(S_{-3}+2F)}}{y^2}+\frac{T^{\omega(S_{-3}+F)}x}{y}.\]
        \end{proposition}

        \begin{remark}
            Recall that the uncorrected SYZ mirror space consists of pairs $(L,\nabla)$, where $L$ is a fiber torus in the SYZ fibration, $\nabla$ is a unitary rank $1$ local system on $L$. In the proposition, $(x,y)\in(\Lambda^*)^2$ is a coordinate for $(L,\nabla)$, defined as
            \begin{align*}
                x=&\ T^{\omega(\beta_1)}\mathrm{hol}_\nabla(\partial\beta_1),\\
                y=&\ T^{\omega(\phi-\beta_2)}\mathrm{hol}_\nabla(-\partial\beta_2),
            \end{align*}
            where $\mathrm{hol}_{\nabla}$ stands for the holonomy of $\nabla$.
        \end{remark}

        \begin{proof}
            Consider the deformation of $J$ constructed in \cite[\S 3.2]{DA09}. The leftmost fiber $F_0$ deforms with $S$ into a holomorphic sphere. Hence, for this perturbation of $J$, there are no walls.
            We denote this perturbation by $J_0(t)$, and denote the perturbation we take in the proposition by $J_z(t)$.

            Take a continuous family of perturbations $\{J_\tau(t)\}$ connecting $J_0(t)$ and $J_z(t)$. The wall appears and moves from $F_0$ to $F_z$ when varying the perturbation $J_\tau(t)$. There are no bubbles appearing away from the wall in this process. Since there is a cobordism between the respective moduli spaces, the count of index $2$ discs in corresponding classes should match for the two choices of perturbation. Therefore, the proposition follows from the formula in \cite[\S 3.2]{DA09}.
        \end{proof}

        Similarly, one can prove the following.

        \begin{proposition}
            If the torus $L$ lies on the left side of $F_z$, the superpotential for $\F_3$ with respect to the chosen perturbation has the form
            \[W_{\mathrm{left}}(x,y)=x+y+\frac{T^{\omega(S_{-3}+3F)}}{xy^3}+\frac{T^{\omega(F)}}{y}+2\frac{T^{\omega(S_{-3}+2F)}}{y^2}+\frac{T^{\omega(2S_{-3}+4F)}}{xy^4}.\]
        \end{proposition}

        $W_\mathrm{right}$ and $W_\mathrm{left}$ represent the superpotential function under two different coordinates. These two expressions are related by a change of variables, which is determined by the contribution of index $0$ discs bounded by tori on the wall. This change of variables is called the \textbf{wall-crossing transformation}. 

        \begin{proposition}\label{change of variables}
            The wall-crossing transformation that relates
            $W_{\mathrm{right}}(x,y)$ and $W_{\mathrm{left}}(x',y')$ is
        \begin{align*}
            y'&=y,\\
            x'&=x\cdot\left(1+\frac{T^A}{y}\right).
        \end{align*} 
            Here $A$ denotes $\omega(S_{-3}+F)$.
        \end{proposition}

        \begin{proof}
            Denote $\omega(F)$ by $B$. 
            By Proposition \ref{index 0 discs for F_3}, the classes of index $0$ discs contributing to the change of variables are $m(\beta_2+\sigma)$. The class $m(\beta_2+\sigma)$ corresponds to the monomial $(\frac{T^A}{y})^m$. From \cite[\S 3.3]{DA07}, we see that $y'=y$ and that
            \[x'=x\cdot h\left(\frac{T^A}{y}\right),\]
            where $h\left(\frac{T^A}{y}\right)$ is a power series of the form $1+O\left(\frac{T^A}{y}\right)$. Thus, we have 
            \begin{align*}
                W_\mathrm{left}(x',y')=W_\mathrm{left}(xh\left(\frac{T^A}{y}\right),y)&=x h\left(\frac{T^A}{y}\right)+y+\frac{T^{A+2B}}{xy^3 h\left(\frac{T^A}{y}\right)}+\frac{T^B}{y}+2\frac{T^{A+B}}{y^2}+\frac{T^{2A+2B}}{xy^4h\left(\frac{T^A}{y}\right)}.
            \end{align*}
            Comparing this expression with that of $W_\mathrm{right}(x,y)$ (they should equal), we have
            \[x h\left(\frac{T^A}{y}\right)+\frac{T^{A+2B}}{xy^3 h\left(\frac{T^A}{y}\right)}+\frac{T^{2A+2B}}{xy^4h\left(\frac{T^A}{y}\right)}=x+\frac{T^{A+2B}}{xy^3}+\frac{T^Ax}{y}.\]
            Viewing both sides as Laurent polynomials of $x$ and comparing the $x$-term, we conclude that
            \[h\left(\frac{T^A}{y}\right)=1+\frac{T^A}{y}.\]
        \end{proof}

        Recall from \cite{DA07} that the mirror is constructed by gluing local charts via coordinate change formulas given by wall-crossing transformations. Thus, by Proposition \ref{change of variables} we have
        \[\F_3^\vee=\{(u,v,w)\in\Lambda^2\times\Lambda^*|\ uv=1+T^Aw\}.\]
        Here we view $(x',y)$ as the coordinate $(u,w^{-1})\in\Lambda\times\Lambda^*$ for this mirror space, and $(x,y)$ as the coordinate $(v^{-1},w^{-1})\in\Lambda^*\times\Lambda^*$. These two coordinates satisfy the desired change of variables formula on the overlap of two charts. Consequently, $W_\mathrm{right}$ and $W_\mathrm{left}$ glue to an analytic function $W$ on $\F_3^\vee$, which is the superpotential then.

        Also, note that the mirror space $(\Lambda^*)^2$ obtained through the perturbation in \cite[\S 3]{DA09} can be viewed as an affine open subspace—the $(x,y)$-chart of $\F_3^\vee$, and the superpotential on $(\Lambda^*)^2$ is the restriction of the superpotential on $\F_3^\vee$.  
        
        In conclusion, we see from the example of $\F_3$ that different choices of perturbation of a non-Fano toric manifold lead to different SYZ mirrors. If we perturb the complex structure in an equivariant manner, the mirror space is still $(\Lambda^*)^n$, and no walls are present. If we perturb in a non-toric way as we do here, the wall-crossing phenomenon occurs, and the new mirror space contains $(\Lambda^*)^n$ as an affine open subspace.

    \section{An SYZ Mirror of the Fourth Hirzebruch Surface $\F_4$}

        In this section, we choose a particular perturbation to investigate the SYZ mirror of $\F_4$. The perturbation we choose here is toric, so the mirror space is again $(\Lambda^*)^2$, which we will see later. The superpotential now has infinitely many terms, as in this case there are infinitely many classes of Maslov index $2$ discs that contribute non-trivially.

        We first define our deformation of the complex structure  of $\F_4$ as follows.

        Recall that $\F_4$ and $\F_2$ are projectivization of bundles $\OO_{\P^1}\oplus\OO_{\P^1}(4)$ and $\OO_{\P^1}\oplus\OO_{\P^1}(2)$ respectively. Take a holomorphic branched double cover $\psi:\P^1\to\P^1$. Since $\OO_{\P^1}\oplus\OO_{\P^1}(2)$ pulls back to $\OO_{\P^1}\oplus\OO_{\P^1}(4)$ along $\psi$, $\psi$ is lifted to a holomorphic bundle map\[\Psi:\F_4\to\F_2.\]Thus, every deformation of $\F_2$ induces a deformation of $\F_4$ via the map $\Psi$. The induced deformation of  $\F_4$ satisfies that $\Psi$ remains holomorphic after deformation.

        We take $\psi:z\mapsto z^2$, so the branched fibers of $\F_4$ under the map $\Psi$ are $F_0$ and $F_\infty$, i.e., the leftmost fiber and the rightmost fiber. Next we take the deformation $J_{\F_2}(t)$ of $\F_2$ constructed in \cite[\S 3.2]{DA09} (in fact the first-order deformation of $\F_2$ is unique up to a scalar). This gives our deformation $J_{\F_4}(t)$ of $\F_4$.

        In the following subsections, we denote again by $\sigma,\phi\in\pi_2(\F_4)$ the homotopy classes of the exceptional section $S_{-4}$ and the fiber of $\F_4$ respectively.

        \subsection{Obstruction Bundle}

            In the deformation $J_{\F_2}(t)$ of $\F_2$, when $t\neq 0$, the deformed $\F_2$ is in fact $\F_0(=\P^1\times\P^1)$ as a complex manifold. For each fiber sphere of $\F_2$, its union with $S_{-2}$ deforms into a holomorphic sphere in $\F_0$. Indeed, these stable spheres are already regular (the obstruction bundle over the corresponding moduli space is the zero bundle).

            Hence, in the deformation $J_{\F_4}(t)$ from $\F_4$ to $\F_0$, the union of the spheres $F_z$, $F_{-z}$ and $S_{-4}$ deforms into a regular holomorphic sphere in $\F_0$ for each $z\in\C^*$, since the nodal sphere $F_z\cup F_{-z}\cup S_{-4}$ is the preimage of the corresponding deforming nodal sphere in $\F_2$. Similarly, two copies of $F_0$ and $S_{-4}$ deform into a holomorphic sphere; two copies of $F_\infty$ and $S_{-4}$ deform into a holomorphic sphere.
            We denote the sphere deformed from $F_z+F_{-z}+S_{-4}$ by $E_z$.

            In this subsection, we show by obstruction bundle that the above described $E_z$ ($z\in\P^1$) are the only regular spheres deformed from nodal spheres of the form `$F_{z_0}+F_{z_1}+S_{-4}$'. (We use `$+$' instead of `$\cup$' because $F_{z_0}$ and $F_{z_1}$ can be the same sphere.) For the rest of the section, we denote $S_{-4}$ by $S$ when there is no risk of ambiguity.

            \begin{remark}
                The reason for considering two fibers with $S$ instead of one fiber with $S$ is that the expected complex dimension of the corresponding moduli space is $1$ for the former and negative for the latter. These deformed spheres $E_z$ will be used in an intersection number argument when finding the classes of index $0$ discs.
            \end{remark}

            Consider the moduli space $\MM(\F_4;\sigma+2\phi)$ consisting of stable spheres $F_{z_0}+F_{z_1}+S$. We identify $\MM(\F_4;\sigma+2\phi)$ with $\P^1\times\P^1$ via the coordinates $z_0,z_1$ of nodal points. Denote the projections of $\P^1\times\P^1$ to each factor by $\pi_0$ and $\pi_1$.

            \begin{remark}
                We point out here that $\MM(\F_4;\sigma+2\phi)$ does not consist of all stable maps from a possibly nodal sphere in the class $\sigma+2\phi$. Rather, it is one of the two irreducible components of the corresponding full moduli space. The other higher dimensional component consists of configurations in the form $\Tilde{F}_z+S$, where $\Tilde{F}_z$ stands for a double branched cover of $F_z$.

                Another difference of $\MM(\F_4;\sigma+2\phi)$ from the usual moduli space is that we did not quotient the space by automorphisms that interchange the two fiber spheres.
            \end{remark}

            \begin{proposition}\label{Ob for F_4}
                The obstruction bundle $\OO b$ over $\MM(\F_4;\sigma+2\phi)\cong\P^1\times\P^1$ is isomorphic to $\pi_0^*\OO_{\P^1}(1)\otimes\pi_1^*\OO_{\P^1}(1)$.
            \end{proposition}

            \begin{proof}
                Using the same method in Proposition \ref{Ob for F_3}, one can show that the $H^1$ of the normal sheaf of the stable sphere $F_{z_0}+F_{z_1}+S$ is $H^1(S,\OO_S(-4)(z_0+z_1))$. Here $z_0+z_1$ represents the sum of divisors $z_0$ and $z_1$ in $S$. Thus, the obstruction bundle $\OO b$ is a line bundle over $\P^1\times\P^1$ with fiber $H^1(\P^1,\OO_{\P^1}(-4)(z_0+z_1))$ over $(z_0,z_1)$.

                Then, the dual $\OO b^\vee$ has fiber $H^0(\P^1,\OO_{\P^1}(2)(-z_0-z_1))$, and can hence be represented as
                \[\OO b^\vee\cong\pi_{01,*}((\pi_2^*\OO_{\P^1}(2))(-\Delta_0-\Delta_1)).\]
                Here $\pi_{01}:\P^1\times\P^1\times\P^1\to\P^1\times\P^1$ is the projection to the first two factors, where $z_0$ and $z_1$ live. $\pi_2:\P^1\times\P^1\times\P^1\to\P^1$ is the projection to the third factor. $\Delta_0$ and $\Delta_1$ are divisors in $\P^1\times\P^1\times\P^1$ defined by $\{z_0=z_2\}$ and $\{z_1=z_2\}$.

                Now since $\OO_{\P^1\times\P^1\times\P^1}(\Delta_0)\cong\pi_0^*\OO_{\P^1}(1)\otimes\pi_2^*\OO_{\P^1}(1)$ and $\OO_{\P^1\times\P^1\times\P^1}(\Delta_1)\cong\pi_1^*\OO_{\P^1}(1)\otimes\pi_2^*\OO_{\P^1}(1)$, we can obtain the expression of $\OO b^\vee$ and the proposition follows.
            \end{proof}

            Similarly, one can prove the counterparts of Proposition \ref{deformation gives a section} and Proposition \ref{zero set of section}. In other words, a deformation of the complex structure on $\F_4$ gives rise to a holomorphic section $s$ of the obstruction bundle $\OO b\to\MM(\F_4;\sigma+2\phi)$, and the zero set $s^{-1}(0)$ consists of stable spheres $u\in\MM(\F_4;\sigma+2\phi)$ that deform into a possibly nodal sphere. Note that the section $s$ has to be symmetric (invariant under the automorphism of $\P^1\times\P^1$ that interchanges $z_0$ and $z_1$).

            We take the deformation to be $J_{\F_4}(t)$. Since we have previously seen that $F_z+F_{-z}+S$ deforms into a holomorphic sphere $E_z$, the anti-diagonal divisor $\{(z,-z)|z\in\P^1\}$ is contained in the zero set $s^{-1}(0)$. However, since $s$ is a nonzero holomorphic section of $\pi_0^*\OO_{\P^1}(1)\otimes\pi_1^*\OO_{\P^1}(1)$, the zero set $s^{-1}(0)$ should have a unique intersection point with each $\P^1$-slice of $\P^1\times\P^1$. Therefore, we conclude that $s^{-1}(0)=\{(z,-z)|z\in\P^1\}$.

            \begin{remark}
                One can prove that $s\neq 0$ by the fact that $s$ is the image of a constant section of $\OO b'$, where $\OO b'$ is rank $3$ trivial bundle with fiber $H^1(S,\OO_S(-4))$, which has a natural surjective bundle map to $\OO b$. However, our discussion in the following subsections does not actually use the fact that $s^{-1}(0)=\{(z,-z)|z\in\P^1\}$.
            \end{remark}

        \subsection{Maslov Index $0$ Holomorphic Discs and Walls}

            In this subsection, we show that there are no regular index $0$ discs in our chosen perturbation $J_{\F_4}(t)$, so no walls are present.

            Fix a product torus $L$ in $\F_4$ (a non-degenerate orbit of the torus action). The homotopy group $\pi_2(\F_4,L)$ is generated by classes $\beta_1,\beta_2,\sigma,\phi$. Here $\beta_1,\beta_2$ are the basic classes indicated in Figure \ref{image_4}.

            \begin{figure}[htbp]  
            \centering  
            \includegraphics[width=0.6\textwidth]{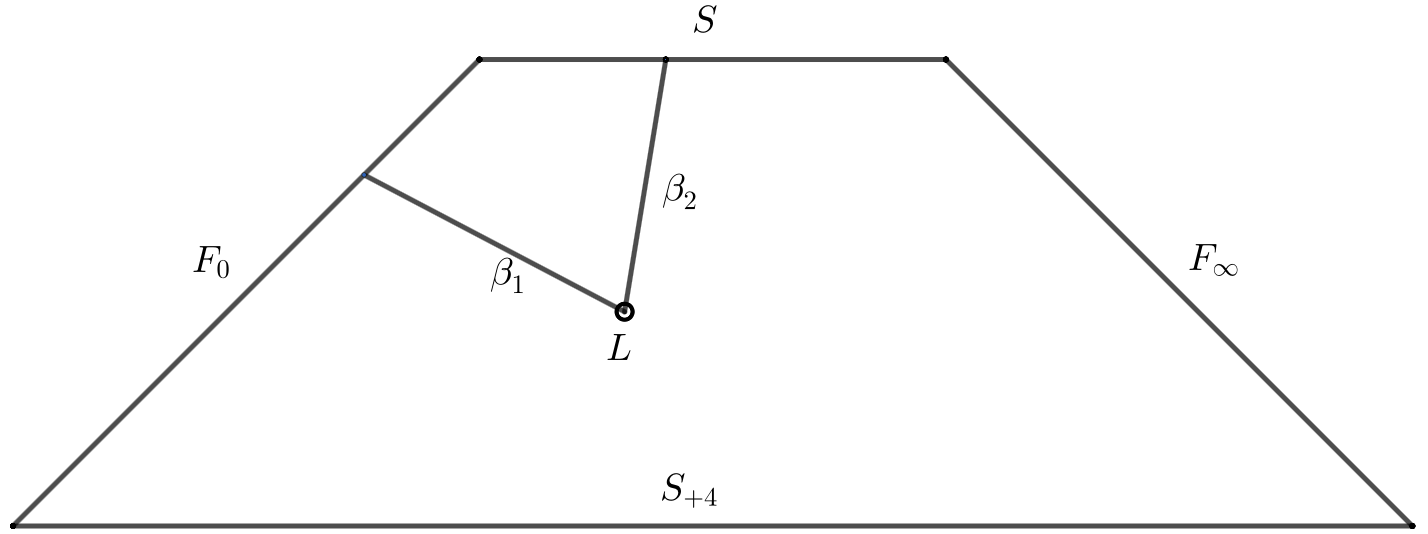}  
            \caption{Disc classes in $\pi_2(\F_4,L)$}  
            \label{image_4}  
        \end{figure}

            \begin{proposition}\label{index 0 discs for F_4}
                The only classes that may contain any index $0$ discs after perturbation are $m(2\beta_2+\sigma)$ for positive integers $m$.
            \end{proposition}

            \begin{proof}
                Suppose $L$ bounds an index $0$ stable disc in the class $m_1\beta_1+m_2\beta_2+m_3\sigma+m_4\phi$, and it deforms into a holomorphic disc after perturbation. It should have non-negative intersection with the spheres $F_0,F_\infty,S_{+4},E_0,E_\infty$ ($E_z$ is the sphere deformed from $F_z+F_{-z}+S$). By a computation of intersection number, we have
            \begin{align*}
                m_1+m_3&\geq0,\\
                m_3&\geq0,\\
                m_4&\geq0,\\
                m_2-2m_3+m_4&\geq0,\\
                2m_1+m_2-2m_3+m_4&\geq0,\\
                m_1+m_2-2m_3+2m_4&=0,
            \end{align*}
            where the last equality comes from the Maslov index condition.

                From these relations, the only possible classes are $2m\beta_2+m\sigma$, $m\in\Z_{>0}$. 
            \end{proof}
            
            In order to exclude the possibility of $m(2\beta_2+\sigma)$ to contain eligible index $0$ discs, we need to examine the chosen deformation more carefully. We start by examining the deformation of $\F_2$ to $\F_0\cong\P^1\times\P^1$ described in \cite{DA09}.

            We choose a coordinate $(z_1',z_2')$ for $\P^1\times\P^1$, the deformed $\F_2$, in the same way as in \cite[\S 3.2]{DA09} (the notation for $(z_1',z_2')$ there is `$(x,y)$'). The fiber of the projection map $\F_2\to\P^1$ over $z\in\P^1$ is denoted by $F_z'$.  These fibers $F_z'$ are regular holomorphic spheres and survive in the deformation. As in \cite[\S 3.2]{DA09}, the deformed fibers $F_z'$ in $\P^1\times\P^1$ are given by $\{z_1'=z\}$, as long as we choose the coordinates in a suitable way. In particular, the branched fibers $F'_0$ and $F'_\infty$ deform into $\{z'_1=0\}$ and $\{z'_1=\infty\}$. 
            
            We define the coordinate $(z_1,z_2)$ on $\P^1\times\P^1$, the deformed $\F_4$, as the pullback of the coordinate $(z_1',z_2')$ under the map $\Psi$. In these coordinates, $\Psi$ has the form
                \[\Psi:(z_1,z_2)\mapsto(z_1',z_2')=(z_1^2,z_2).\]

            Below, we summarize the deformations of the divisors that we are interested in. When deforming the complex structure away from that of $\F_4$ (when $t\neq 0$):
            \begin{itemize}
                \item The exceptional section $S_{-4}$ disappears.
                \item The fibers $F_z$ deform into spheres $\{z_1=z\}$.
                \item $S_{+4}$ deforms into the sphere $\{z_1^2z_2=\epsilon\}$, since $S_{+4}$ is the preimage of $S_{+2}$ under $\Psi$, and $S_{+2}$ deforms into $\{z_1'z_2'=\epsilon\}$. Here $\epsilon$ is a real number which depends on $t$ (see \cite[\S 3.2]{DA09}).
                \item $S_{-4}$ and two copies of $F_0$ deform into a holomorphic sphere $E_0=\{z_2=\infty\}$, since $E_0$ is the preimage of $E_0'=\{z_2'=\infty\}$, the sphere deformed from $S_{-2}\cup F'_0$ in $\F_2$.
            \end{itemize}
            
            Now we note that, after deformation, the divisor $E_0+F_\infty+S_{+4}$ is no longer anticanonical, since $E_0$ contains two copies of the fiber class.

            We denote the deformed $S_{+4}$ by $D_\epsilon=\{z_1^2z_2=\epsilon\}$, and denote the divisors $\{z_2=0\},\{z_1=0\},\{z_2=\infty\},\{z_1=\infty\}$ by $A_0,B_0,A_\infty,B_\infty$ respectively. In this way, the anticanonical divisor can be chosen as
            \[D_\epsilon-B_0+A_\infty+B_\infty,\]
            \[\text{or}\ \ \ D_\epsilon+\frac{1}{2}A_0+\frac{1}{2}A_\infty.\]
            (The reason we consider these specific divisors is the property stated in Proposition \ref{index formula} below, where we calculate the Maslov index using these divisors.)

            So far, we have discussed the deformation of the complex structure and related divisors. The more challenging aspect compared to the case of $\F_3$ lies in the need to also deform the K\"ahler form and the Lagrangian tori, which is required to further analyze the holomorphic discs.
            
            Now we consider the tori.  We take the tori in the deformed $\F_4$ to be the preimages of the Chekanov tori 
            \begin{equation*}
                \left\{
            \begin{array}{ll}
                |z'_1 z'_2-\epsilon|=r,\\
                |z'_1|^2-|z'_2|^2=\lambda, \\
            \end{array}
            \right.
            \end{equation*}
            in the deformed $\F_2$. In other words, we define a torus $T_{r,\lambda}$ by equations
            \begin{equation}\label{naive torus family}
                \left\{
            \begin{array}{ll}
                |z_1^2z_2-\epsilon|=r,\\
                |z_1|^4-|z_2|^2=\lambda, \\
            \end{array}
            \right.
            \end{equation}
            where $0<r<|\epsilon|$ and $\lambda\in\R$. When $r$ and $\lambda$ vary, this gives a family of tori.

            \begin{remark}
                The deformation of $\F_2$ to $\F_0$ in \cite[\S 3.2]{DA09} is compatible with an $S^1$-action which becomes $(z_1',z_2')\mapsto(e^{\mathrm{i}\theta}z_1',e^{-\mathrm{i}\theta}z_2')$ on $\P^1\times \P^1$. One deforms product tori in $\F_2$ to $S^1$-invariant Lagrangian tori in $\P^1\times\P^1$ (equipped with an $S^1$-invariant K\"ahler form), namely
                \begin{equation*}
                \left\{
            \begin{array}{ll}
                |z'_1 z'_2-\epsilon|=r,\\
                \mu_{S^1}(z_1',z_2')=\lambda/2, \\
            \end{array}
            \right.
            \end{equation*}
                where $\mu_{S^1}$ is the moment map of the $S^1$-action. Lifting by $\Psi$ to the branched double cover (also equipped with an $S^1$-invariant K\"ahler form), we obtain Lagrangian tori, which are invariant under the $S^1$-action $(z_1,z_2)\mapsto(e^{\mathrm{i}\theta}z_1,e^{-2\mathrm{i}\theta}z_2)$, defined by equations \begin{equation*}
                \left\{
            \begin{array}{ll}
                |z_1^2z_2-\epsilon|=r,\\
                \mu_{S^1}(z_1^2,z_2)=\lambda/2, \\
            \end{array}
            \right.
            \end{equation*}
                The equations \eqref{naive torus family} correspond to a specific choice of the K\"ahler form which we explain below.
            \end{remark}

            To get a better sense of how these tori look like, we consider the map 
            \begin{align*}
                f:\C\times\C&\to\C,\\
                (z_1,z_2)&\mapsto z_1^2z_2,
            \end{align*}
            Here we view $\C\times\C$ as the complement of divisors at infinity in $\P^1\times\P^1$. $f$ is a fibration with general fibers $\{z_1^2z_2=c\neq 0\}\cong\C^*$ and a singular fiber $\{z_1^2z_2=0\}$. Then, the torus $T_{r,\lambda}$ defined by \eqref{naive torus family} lives over the circle $\{|z-\epsilon|=r\}$ in $\C$ and intersects each fiber of $f$ (over $\{|z-\epsilon|=r\}$) at a circle.
            Since $r<|\epsilon|$, the circle $\{|z-\epsilon|=r\}$ does not encompass the origin.

            \begin{figure}[htbp]  
            \centering  
            \includegraphics[width=0.8\textwidth]{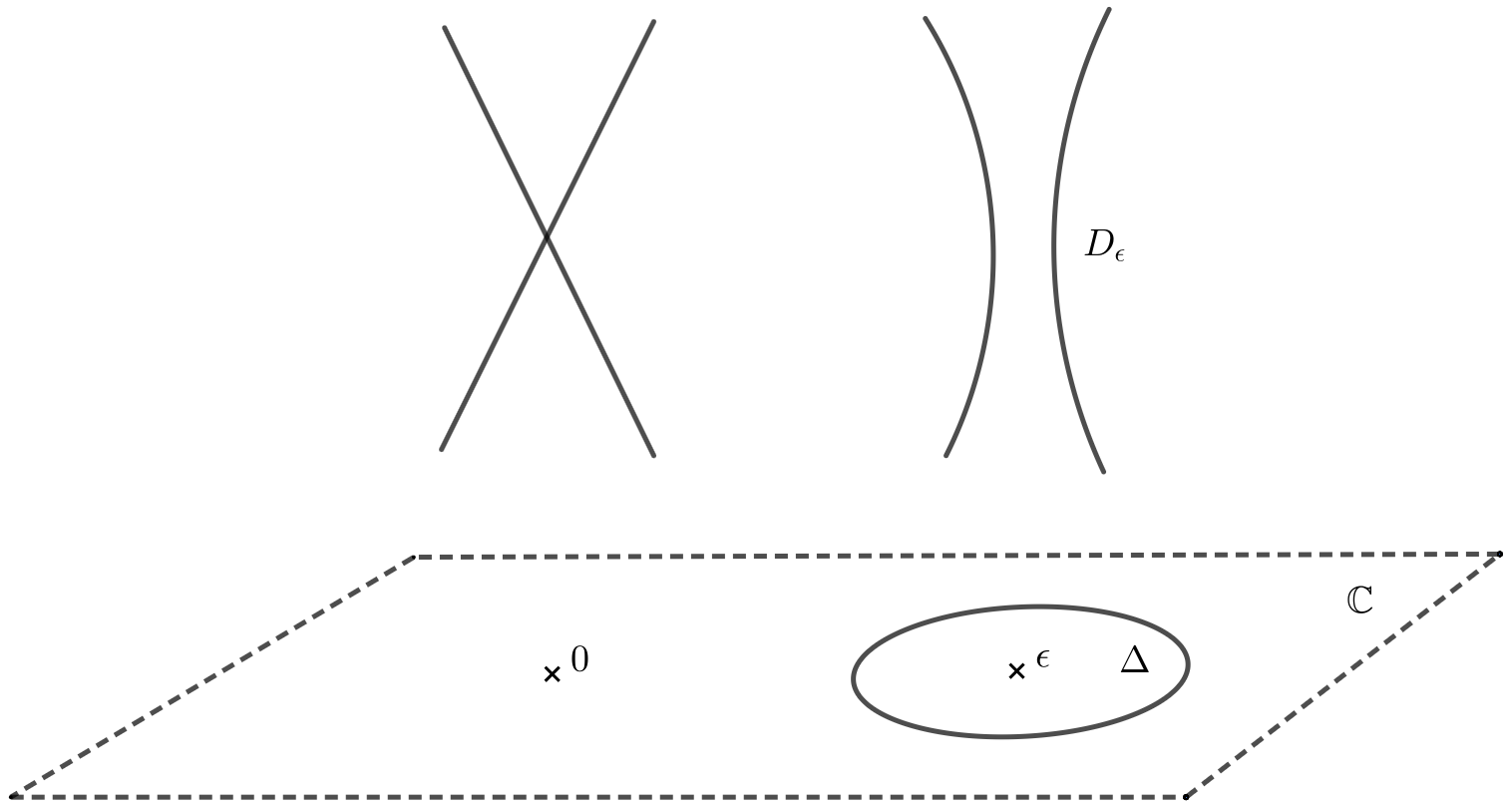}  
            \caption{The map $f:\C\times\C\to\C$ as a fibration}  
            \label{image_5}  
        \end{figure}

            Next, we construct a K\"ahler form on $\P^1\times\P^1$, aiming at making the tori $T_{r,\lambda}$ Lagrangian. A natural attempt is to take the pullback of the K\"ahler form $\omega_\epsilon$ on the deformed $\F_2$ described in \cite[\S 3.2]{DA09}. However, this pullback form $\Psi^*\omega_\epsilon$ degenerates at the branched fibers $B_0=\{z_1=0\}$ and $B_\infty=\{z_1=\infty\}$, since tangent vectors $\partial_{z_1}$ and $\partial_{\bar{z}_1}$ lie in the kernel of $\Psi^*\omega_\epsilon$. Hence, we add a small perturbation term to $\Psi^*\omega_\epsilon$ near $B_0$ and define the form as
            \begin{equation}\label{modified kahler form}
                \Psi^*\omega_\epsilon+\mathrm{i}\delta\partial\bpar(\chi(|z_1|)\mathrm{log}(1+|z_1|^2)).
            \end{equation}
            Here, $\delta>0$ is a small real number, $\chi:\R_{\geq 0}\to[0,1]$ is a smooth bump function such that
            \begin{align*}
                \chi&=1\ \text{on}\ [0,\delta'/2],\\
                \chi&=0\ \text{on}\ [\delta',+\infty),
            \end{align*}
            and $\delta'>0$ is another small real number. We can see that the form \eqref{modified kahler form} is indeed a K\"ahler form on the complement of infinity. Because the additional term takes positive value on $\partial_{z_1},\partial_{\bar{z}_1}$ near zero, and on $|z_1|\in[\delta'/2,\delta']$ this form remains positive by choosing $\delta$ to be small enough. (Moreover, this perturbation preserves the $S^1$-invariance of the K\"ahler form.)

            Likewise, we add a similar term near $B_\infty$, and hence get a K\"ahler form for which the tori $T_{r,\lambda}$ away from supports of the two chosen bump functions are Lagrangian (after a further technical modification of the K\"ahler form, see the second remark in the following).

            \begin{remark}
                For our purpose, we do not need the tori in the form $T_{r,\lambda}$ to be all Lagrangian. Instead, we fix a torus $T_{r,\lambda}$ (fix $r,\lambda$), and investigate holomorphic discs with boundary on this fixed torus. Thus, by taking the supports of the bump functions small enough, we make the torus $T_{r,\lambda}$ Lagrangian. Note that, in this way, the tori $T_{r,\lambda'}$ with $|\lambda'|\leq|\lambda|$ are Lagrangian as well.
            \end{remark}

            \begin{remark}
                A technical issue here is that the K\"ahler form $\omega_\epsilon$ on the deformed $\F_2$ constructed in \cite[\S 3.2]{DA09} does not actually make the Chekanov tori defined by $\{|z_1'z_2'-\epsilon|=r,\ |z_1'|^2-|z_2'|^2=\lambda\}$ Lagrangian. Rather, the tori defined by $\{|z_1'z_2'-\epsilon|=r,\ \mu_{S^1}(z_1',z_2')=\lambda/2\}$ are Lagrangian with respect to $\omega_\epsilon$, where $\mu_{S^1}$ is the moment map of the $S^1$-action $(z_1',z_2')\mapsto(e^{\mathrm{i}\theta}z_1',e^{-\mathrm{i}\theta}z_2')$ ($\omega_\epsilon$ is invariant under this $S^1$-action). 
                
                To address this issue, there is a deformation of K\"ahler forms $\{\omega_\epsilon^{(\tau)},\tau\in[0,1]\}$ satisfying $\omega_\epsilon^{(0)}=\omega_\epsilon$, and that the moment map $\mu^{(1)}_{S^1}$ of the $S^1$-action for $\omega_\epsilon^{(1)}$ coincides with $(z_1',z_2')\mapsto (|z_1'|^2-|z_2'|^2)/2$ for $(z_1',z_2')$ such that \[|z_1'z_2'-\epsilon|\leq r_0,\ \ -\lambda_0\leq|z_1'|^2-|z_2'|^2\leq\lambda_0.\]
                Here $r_0$ is an arbitrary fixed positive number smaller than $|\epsilon|$, and $\lambda_0$ is an arbitrary fixed positive number. In short, $\omega^{(1)}_{\epsilon}$ has the standard, desired moment map over a sufficiently large domain in $\P^1\times\P^1$, making the Chekanov torus we are considering Lagrangian.
                
                For each $\tau\in[0,1]$, $\omega_\epsilon^{(\tau)}$ induces a K\"ahler form on the deformed $\F_4$ as we explained before (by pulling back and perturbing), as well as a family of Lagrangian tori
                \begin{equation*}
                \left\{
            \begin{array}{ll}
                |z_1^2z_2-\epsilon|=r,\\
                \mu^{(\tau)}_{S^1}(z_1^2,z_2)=\lambda/2, \\
            \end{array}
            \right.
            \end{equation*}for $(z_1,z_2)$ in a sufficiently large domain in $\P^1\times\P^1$.

                One can prove the following Proposition \ref{index formula} and Lemma \ref{no index 0 disc for F_4} for each $\omega_{\epsilon}^{(\tau)}$ and see that there is no index $0$ disc bubble in the deformation process. Therefore, we can use the K\"ahler form induced from $\omega^{(1)}_\epsilon$, under which the torus $T_{r,\lambda}$ defined by \eqref{naive torus family} is Lagrangian.
            \end{remark}

            Finally, we return to consider the discs with boundary on $T_{r,\lambda}$. We start with an index formula for the disc classes.

            \begin{proposition}\label{index formula}
                For a class $\beta\in\pi_2(\P^1\times\P^1,T_{r,\lambda})$, the Maslov index of $\beta$ satisfies
                \begin{align*}
                    \mu(\beta)&=2\beta\cdot(D_{\epsilon}-B_0+A_\infty+B_\infty)\\
                    &=2\beta\cdot(D_{\epsilon}+\frac{1}{2}A_0+\frac{1}{2}A_\infty).
                \end{align*}
            \end{proposition}

            \begin{proof}
                $\pi_2(\P^1\times\P^1,T_{r,\lambda})$ has a basis $\alpha_0,\beta_0,A,B$. Here $A$ denotes the class of $A_0$ (or equivalently $A_\infty$) and $B$ denotes the class of $B_0$. $\beta_0$ is the class of sections of $f$ over the disc $|z-\epsilon|\leq r$. $\alpha_0$ is the class of Lefschetz thimble associated to the critical value of $f$ (which is the preimage of the Lefschetz thimble described in \cite[\S 5.2]{DA07} under $\Psi$). 
                
                We only need to show that the Maslov index of these basis classes coincide with the corresponding intersection number, which is indeed the case. Results of our computation of intersection and index $\mu$ are summarized in Table \ref{intersection numbers} below, where the empty cells represent zero. Here, the way one computes the Maslov index of a disc is to track the tangent space of the disc along the boundary circle and find the Maslov index of a loop of totally real subspaces.
                \begin{table}[h]
                \centering
                \begin{tabular}{c|ccccc|c}
                    \hline
                     & $A_0$ & $B_0$ & $A_\infty$ & $B_\infty$ & $D_\epsilon$ & $\mu$\\
                    \hline
                    $\alpha_0$ & $-2$ & $1$ & & & & $-2$\\
                    \hline
                    $\beta_0$ & & & & & $1$ & $2$\\
                    \hline
                    $A$ & & $1$ & & $1$ & $2$ & $4$\\
                    \hline
                    $B$ & $1$ & & $1$ & & $1$ & $4$\\
                    \hline
                \end{tabular}
                \caption{Intersection numbers and Maslov indices of classes in $\pi_2(\P^1\times\P^1,T_{r,\lambda})$}
                \label{intersection numbers}
                \end{table}
            \end{proof}

            As mentioned earlier, the divisor $D_\epsilon$ corresponds to $S_{+4}$ in the original picture of $\F_4$ (Figure \ref{image_4}), so $S_{+4}$ is in the class $A+2B$ (by topologically identifying $\F_4$ with $\P^1\times\P^1$). Similarly, $S_{-4}$ is in the class $A-2B$, and the fiber is in the class $B$. In other words, $\sigma=A-2B$, and $\phi=B$. $A,B$ are just substitutes for the notations of these classes.

            Likewise, we also see the disc class $\beta_0$ equals the class $B-\beta_2$ in the original picture of $\F_4$, i.e., $\beta_0$ is another basic class of $\F_4$. This is because a class in $\pi_2(\P^1\times\P^1,T_{r,\lambda})$ is uniquely determined by its intersection numbers with $D_\epsilon,A_0,B_0,A_\infty,B_\infty$, as we can see from Table \ref{intersection numbers}.

            Having matched the homotopy classes in $\pi_2(\P^1\times\P^1,T_{r,\lambda})$ with classes in $\F_4$, we can now look back at index $0$ discs. Proposition \ref{index 0 discs for F_4} tells us that the only classes that may contain index $0$ discs after deformation are $m(2\beta_2+\sigma)=m(A-2\beta_0)$ for $m>0$. This can also been seen from Table \ref{intersection numbers} by applying the argument of positive intersection. 

            Suppose $u:(D^2,\partial D^2)\to(\P^1\times\P^1,T_{r,\lambda})$ is a holomorphic disc in the class $m(A-2\beta_0)$. Since $u$ does not intersect with $D_\epsilon$, $u$ lives over the complement of the disc $\Delta:=\{|z-\epsilon|\leq r\}$ in Figure \ref{image_5} of the fibration $f$. Note that $f:\C\times\C\to\C$ extends to a map
            \begin{align*}
                f:\P^1\times\P^1\setminus \{(0,\infty),(\infty,0)\}&\to\P^1,\\
                (z_1,z_2)&\mapsto z_1^2z_2.
            \end{align*}
            Since $u$ does not intersect $A_0$ or $A_\infty$, $u$ does not pass through $(\infty,0)$ or $(0,\infty)$.

            Define $\bar{u}$ to be the composition $f\circ u$. The fact that $u$ does not intersect with $D_\epsilon$ implies that $\bar{u}$ is a map to $\P^1\setminus\Delta$. In other words,
            \[\bar{u}=f\circ u:(D^2,\partial D^2)\to(\P^1\setminus\Delta,\partial\Delta).\]

            We view $\bar{u}=u_1^2u_2$ as a meromorphic function on $D^2$, where $u_1$ is the $z_1$-component of $u$ and $u_2$ is the $z_2$-component of $u$. Each zero of $\bar{u}$ corresponds to an intersection point of $u$ and $A_0$ or $B_0$. Each pole of $\bar{u}$ corresponds to an intersection point of $u$ and $A_\infty$ or $B_\infty$. In our case, $u$ has $m$ intersections with $B_0$, $m$ intersections with $B_\infty$, and no intersection with $A_0$ or $A_\infty$. Thus, $\bar{u}$ has exactly $m$ double zeros and $m$ double poles. (We allow higher poles like a quadruple pole, which we count as two double poles.)

            \begin{lemma}\label{no index 0 disc for F_4}
                Such a meromorphic function $\bar{u}:(D^2,\partial D^2)\to(\P^1\setminus\Delta,\partial\Delta)$ with $m$ double zeros and $m$ double poles does not exist.
            \end{lemma}

            \begin{proof}
                $\bar{u}$ is a degree $2m$ holomorphic map from a disc to another disc. By applying the Riemann-Hurwitz formula for a holomorphic map $(\Sigma,\partial\Sigma)\to(\Sigma',\partial\Sigma')$
                \[\chi(\Sigma')=d\chi(\Sigma)-R\]
                ($d$ is the degree of the map and $R$ is the total ramification) to $\bar{u}$, we have $R(\bar{u})=2m-1$, i.e., the sum of ramification indices of $\bar{u}$ is $2m-1$. However, $R(\bar{u})$ is at least $2m$ since $\bar{u}$ has $m$ double zeros and $m$ double poles, which yields a contradiction.
            \end{proof}

            From the lemma above, we conclude that the holomorphic disc $u:(D^2,\partial D^2)\to(\P^1\times\P^1,T_{r,\lambda})$ in the class $m(A-2\beta_0)$ does not exist. Therefore, there are no index $0$ discs in our chosen perturbation, and no walls are present.

        \subsection{Maslov Index $2$ Holomorphic Discs and Superpotential}

            In the previous subsection, we have seen that the SYZ mirror space of $\F_4$ with respect to the perturbation given by $J_{\F_4}(t)$ is $(\Lambda^*)^2$. In this subsection, we determine the superpotential as an analytic function on $(\Lambda^*)^2$, by counting regular index $2$ discs with boundary on $T_{r,\lambda}$.

            Note that we can assume that $\lambda=0$. Since deforming $\lambda$ to $0$ yields a Lagrangian isotopy from $T_{r,\lambda}$ to $T_{r,0}$ (see the first remark preceding Proposition \ref{index formula}, and note that none of these Lagrangians bound index $0$ discs) without intersecting any of the divisors $D_\epsilon,A_0,B_0,A_\infty,B_\infty$, the count of holomorphic discs that satisfy an intersection condition with these divisors stays the same.

            Making use of Proposition \ref{index formula}, we have the following.

            \begin{proposition}\label{index 2 discs for F_4}
                The only classes in $\pi_2(\P^1\times\P^1,T_{r,\lambda})$ that may contain index $2$ discs after perturbation are\begin{align*}
                    \beta_0+m(A-2\beta_0),\\
                    (B-\beta_0)+m(A-2\beta_0),\\
                    (2B+\alpha_0-2\beta_0)+m(A-2\beta_0),\\
                    (A-\alpha_0-2\beta_0)+m(A-2\beta_0),
                \end{align*}
                where $m$ takes value in $\Z_{\geq 0}$.
            \end{proposition}

            \begin{remark}
                Since we have matched the classes in $\pi_2(\P^1\times\P^1,T_{r,\lambda})$ with classes in $\pi_2(\F_4,L)$ in the discussion following Proposition \ref{index formula}, we can see that the classes $\beta_0$, $B-\beta_0$, $2B+\alpha_0-2\beta_0$, and $A-\alpha_0-2\beta_0$ correspond exactly to the four basic classes in $\F_4$.
            \end{remark}

            \begin{proof}
                One can prove the result for corresponding classes in $\pi_2(\F_4,L)$, in the same way as in Proposition \ref{index 0 discs for F_4}. Alternatively, we can make use of Table \ref{intersection numbers} and prove the result directly for classes in $\pi_2(\P^1\times\P^1,T_{r,\lambda})$, as follows.

                Suppose $T_{r,\lambda}$ bounds an index $2$ stable disc in the class $n_1\alpha_0+n_2\beta_0+n_3A+n_4B$, and it deforms into a holomorphic disc after perturbation. It should have non-negative intersection with the spheres $A_0,B_0,A_\infty,B_\infty,D_\epsilon$. We have\begin{align*}
                    -2n_1+n_4&\geq0,\\
                    n_1+n_3&\geq0,\\
                    n_4&\geq0,\\
                    n_3&\geq0,\\
                    n_2+2n_3+n_4&\geq0,\\
                    -n_1+n_2+2n_3+2n_4&=1.
                \end{align*}
                The proposition follows directly from these relations.
            \end{proof}

            Then, we need to find the contribution of these classes to the superpotential, which amounts to finding the count of regular holomorphic discs in these classes. More precisely, the disc count is the number of holomorphic discs in the class whose boundary passes through a generic point in $T_{r,\lambda}$.

            We begin with the classes $(B-\beta_0)+m(A-2\beta_0)$, for which the result is stated in Proposition \ref{disc count}. Before proving that, we need a technical lemma from complex analysis.

            \begin{lemma}\label{bar u}
                There exists a meromorphic function $\bar{u}:(D^2,\partial D^2)\to(\P^1\setminus\Delta,\partial\Delta)$ with exactly $1$ simple zero, $1$ simple pole, $m$ double zeros, and $m$ double poles. $\bar{u}$ is unique up to an automorphism of the domain. ($\Delta$ is a disc in $\P^1$ that does not contain $0$ or $\infty$.)
            \end{lemma}

            \begin{proof}
                \textbf{Claim.} At the topological level, the map $\bar{u}$ exists and is unique. 
                
                In other words, there exists a branched covering map $\bar{u}$ from the topological disc $(D^2,\partial D^2)$ to $(\P^1\setminus\Delta,\partial\Delta)$, such that $\bar{u}^{-1}(0)=\{a_0,\cdots,a_m\}$, $\bar{u}^{-1}(\infty)=\{b_0,\cdots,b_m\}$, and $a_1,\cdots,a_m,b_1,\cdots,b_m$ are (topologically standard) double branched points, giving all the branched points of $\bar{u}$. Moreover, $\bar{u}$ is unique up to a self-homeomorphism of $(D^2,\partial D^2)$.

                We now prove the claim. The existence of $\bar{u}$ in the case of $m=1$ follows from Figure \ref{image_7} below, which uses the cut-and-paste construction to obtain a branched covering map with standard double branched points at $a_1$ and $b_1$. For general $m$, one has a similar picture where $a_0$, $b_m$, $a_1$, $b_{m-1}$, ..., $a_m$, $b_0$ appear sequentially along a curve in $D^2$, and a similar cut-and-paste construction.

                \begin{figure}[htbp]  
            \centering  
            \includegraphics[width=0.7\textwidth]{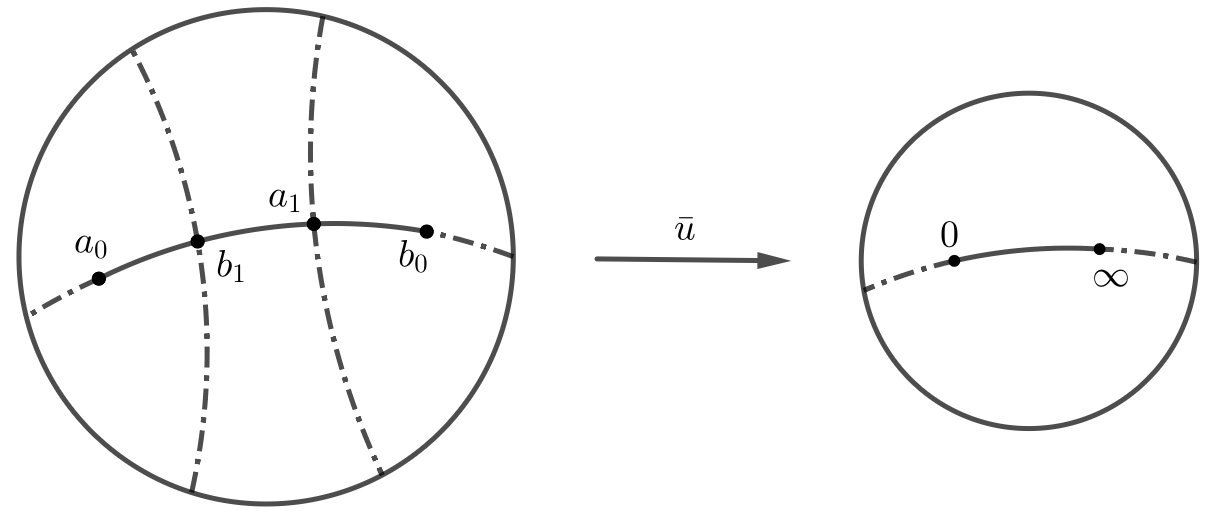}  
            \caption{Construction of the branched covering map $\bar{u}$}  
            \label{image_7}  
        \end{figure}

                For the topological uniqueness of $\bar{u}$, we need the following observation. Connect $0$ and $\infty$ in $\P^1\setminus\Delta$ with a curve $\gamma$. Consider the lift $\bar{u}^{-1}(\gamma)$ of $\gamma$ in $D^2$. $\bar{u}^{-1}(\gamma)$ is a curve with end points given by the simple zero and the simple pole, and $\bar{u}^{-1}(\gamma)$ passes through all the double zeros and double poles, where zeros and poles appear alternatively.
                
                If we have two such maps $\bar{u}_1,\bar{u}_2$, after composing a homeomorphism on the domain, we can make the curves $\bar{u}_1^{-1}(\gamma)$ and $\bar{u}_2^{-1}(\gamma)$ agree as subsets of the domain, and also make the point set $\bar{u}_1^{-1}(0)$ agree with $\bar{u}_2^{-1}(0)$, $\bar{u}_1^{-1}(\infty)$ agree with $\bar{u}_2^{-1}(\infty)$. Elongate the curve $\gamma$ (in $\P^1\setminus\Delta$) beyond $0$ and $\infty$ and get the dotted segments in Figure \ref{image_7}. We can further make the preimages of the dotted segments agree after composing a homeomorphism on the domain. Then, the uniqueness of $\bar{u}$ is already clear.

                We now return to the lemma.

                \textbf{Step 1.} Existence.

                Take a topological branched covering map $\bar{u}$ from the topological disc $(D^2,\partial D^2)$ to $(\P^1\setminus\Delta,\partial\Delta)$, satisfying our conditions on zeros and poles. $\bar{u}$ is an honest covering map restricted to
                \[D^2\setminus\{a_0,\cdots,a_m,b_0,\cdots,b_m\}\to(\P^1\setminus\Delta)\setminus\{0,\infty\}.\]

                By pulling back the standard complex structure on $\P^1\setminus\Delta$, we get a complex structure on $D^2\setminus\{a_1,\cdots,a_m,b_1,\cdots,b_m\}$. We extend this complex structure to the entire $D^2$ as follows. 

                Since $a_1$ is a topologically standard double branched point, there exists topological embeddings $\varphi:D(\delta)\to D^2$ sending $0$ to $a_1$, and $h:D(\delta^2)\to \P^1\setminus\Delta$ sending $0$ to $0$, such that $h^{-1}\circ \bar{u}\circ\varphi(z)=z^2$. Here $D(\delta)$ denotes an open radius $\delta$ disc centered at $0$, equipped with the standard complex structure restricted from $\C$.
                \[\begin{tikzcd}
	{D(\delta)} & {D^2} \\
	{D(\delta^2)} & {\P^1\setminus\Delta}
	\arrow["\varphi", hook, from=1-1, to=1-2]
	\arrow["{z\mapsto z^2}"', two heads, from=1-1, to=2-1]
	\arrow["{\bar{u}}", two heads, from=1-2, to=2-2]
	\arrow["h"', hook, from=2-1, to=2-2]
\end{tikzcd}\]

                By the Riemann mapping theorem, $h$ can be chosen to be a holomorphic map. More precisely, there exists a map $h':D(\delta^2)\to\P^1\setminus\Delta$ sending $0$ to $0$, such that $h'$ is a biholomorphism onto $h(D(\delta^2))$. Then we lift $h'$ to a continuous map $\varphi'$ such that the following diagram commutes (this lifting is a purely topological construction).\[\begin{tikzcd}
	{D(\delta)} & {D^2} \\
	{D(\delta^2)} & {\P^1\setminus\Delta}
	\arrow["\varphi'", hook, from=1-1, to=1-2]
	\arrow["{z\mapsto z^2}"', two heads, from=1-1, to=2-1]
	\arrow["{\bar{u}}", two heads, from=1-2, to=2-2]
	\arrow["h'"', hook, from=2-1, to=2-2]
\end{tikzcd}\]

                \textbf{Claim.} $\varphi':D(\delta)\to D^2$ is a holomorphic chart for $D^2$ which is compatible with the pullback complex structure away from the branched points.

                In fact, the composition $\bar{u}\circ \varphi'=h'\circ (z\mapsto z^2)$ is holomorphic, which implies the compatibility of the chart $(D(\delta),\varphi')$.

                Similarly, there are holomorphic charts around other branched points. These data give a well-defined complex structure of $D^2$ such that $\bar{u}$ is holomorphic. Since the complex structure on the disc $D^2$ is unique, the existence part is proved.

                \textbf{Step 2.} Uniqueness.

                Suppose there are two holomorphic maps $\bar{u}_1,\bar{u}_2:(D^2,\partial D^2)\to(\P^1\setminus\Delta,\partial\Delta)$ satisfying the condition. By the topological uniqueness of such a map, there exists a homeomorphism $h:(D^2,\partial D^2)\to(D^2,\partial D^2)$ such that $\bar{u}_1=\bar{u}_2\circ h$.
                $h$ is biholomorphic away from the branched points. In the following, we prove that $h$ is also biholomorphic near the branched points.

                Suppose that $a,a'\in D^2$ are branched points of $\bar{u}_1$ and $\bar{u}_2$ respectively and $h(a)=a'$. In holomorphic charts centered at $a$ and $a'$, $\bar{u}_1$ has the form $z\mapsto w=z^2$, and $\bar{u}_2$ has the form $z'\mapsto w'=z'^2$. Since $w$ and $w'$ are both coordinates centered at $0$ on $\P^1\setminus\Delta$, there exists a biholomorphism $\varphi$ defined on a neighborhood of $0$ such that $w'=\varphi(w)$, i.e., $z'^2=\varphi(z^2)$.

                Hence, by taking the square root of the function $\varphi(z^2)$, we get a biholomorphism from a neighborhood of $a$ to that of $a'$, which match with $h$. The uniqueness of $\bar{u}$ follows.
            \end{proof}

            Now we return to the disc count for the class $(B-\beta_0)+m(A-2\beta_0)$. We assume $\lambda=0$ for the rest of this subsection.

            \begin{proposition}\label{disc count}
                The torus $T_{r,0}$ bounds a unique $S^1$-family of regular holomorphic discs in the class $(B-\beta_0)+m(A-2\beta_0)$, the disc count of which is $2m+1$.
            \end{proposition}

            \begin{proof}
                \textbf{Step 1.} Uniqueness.
                
                Suppose $u:(D^2,\partial D^2)\to(\P^1\times\P^1,T_{r,\lambda})$ is a holomorphic disc in the class $(B-\beta_0)+m(A-2\beta_0)$. By Table \ref{intersection numbers}, $u$ has $1$ intersection with both $A_0$ and $A_\infty$, and has $m$ intersections with both $B_0$ and $B_\infty$. $u$ does not intersect with $D_\epsilon$, and hence does not pass through $(0,\infty)$ or $(\infty,0)$ where $f$ is not defined. 

                As a result, $u$ lives over $\P^1\setminus\Delta$, and the meromorphic function $\bar{u}=f\circ u=u_1^2u_2:(D^2,\partial D^2)\to(\P^1\setminus\Delta,\partial\Delta)$ has $1$ simple zero $a_0$, $m$ double zeros $a_1,\cdots,a_m$, and $1$ simple pole $b_0$, $m$ double poles $b_1,\cdots,b_m$. Applying the Riemann-Hurwitz formula as in Lemma \ref{no index 0 disc for F_4}, we see that it is not possible for these zeros or poles to merge into higher-order zeros or poles. By Lemma \ref{bar u}, $\bar{u}$ is unique up to an automorphism of the domain.

                Consider another meromorphic function $\Tilde{u}:=u_1^2/u_2$ on $D^2$. Since $|z_1|^4-|z_2|^2=\lambda=0$ on the torus $T_{r,0}$, we have $|\Tilde{u}|=1$ on the boundary $\partial D^2$. We completely understand the zeros and poles of $\Tilde{u}$ since we understand those of $u_1$ and $u_2$.
                Hence, $\Tilde{u}(z)$ must be of the form
                \[e^{\mathrm{4i\theta}}\cdot\left(\frac{z-a_0}{1-\bar{a}_0z}\right)^{-1}\cdot\left(\frac{z-a_1}{1-\bar{a}_1z}\right)^2\cdot\cdots\cdot\left(\frac{z-a_m}{1-\bar{a}_mz}\right)^2\cdot\left(\frac{z-b_0}{1-\bar{b}_0z}\right)^1\cdot\left(\frac{z-b_1}{1-\bar{b}_1z}\right)^{-2}\cdot\cdots\cdot\left(\frac{z-b_m}{1-\bar{b}_mz}\right)^{-2}.\]
                In other words, $\Tilde{u}$ is unique up to a factor $e^{4i\theta}$. We denote by $\Tilde{u}_0(z)$ the expression
                \[\left(\frac{z-a_0}{1-\bar{a}_0z}\right)^{-1}\cdot\left(\frac{z-a_1}{1-\bar{a}_1z}\right)^2\cdot\cdots\cdot\left(\frac{z-a_m}{1-\bar{a}_mz}\right)^2\cdot\left(\frac{z-b_0}{1-\bar{b}_0z}\right)^1\cdot\left(\frac{z-b_1}{1-\bar{b}_1z}\right)^{-2}\cdot\cdots\cdot\left(\frac{z-b_m}{1-\bar{b}_mz}\right)^{-2}.\]

                Then, $u_1$ is a fourth root of $\bar{u}\cdot\Tilde{u}$, i.e.,
                \[u_1=e^{\mathrm{i}\theta}(\bar{u}\Tilde{u}_0)^{1/4},\]
                and $u_2$ is determined by $u_2=\bar{u}/u_1^2$. This determines a unique $S^1$-family of $u=(u_1,u_2)$ parametrized by $\theta$. Note that different choices of the fourth root are reflected in a change in $\theta$.

                \textbf{Step 2.} Existence.
                
                In fact, the existence is already clear from the proof of the uniqueness. We define $\bar{u}$ by Lemma \ref{bar u} and define $\Tilde{u}_0$ as before. Then $u=(u_1,u_2)$ can be defined as before: $u_1=e^{\mathrm{i}\theta}(\bar{u}\Tilde{u}_0)^{1/4}$, $u_2=\bar{u}/u_1^2$. This gives us the desired family of holomorphic discs.

                \textbf{Step 3.} Regularity.

                We need to prove that the above constructed disc is regular, which means the linearized operator \[D_{\bpar_{J},u}:\Omega^0(u^*T(\P^1\times\P^1),u^*T(T_{r,0}))\to\Omega^{0,1}(u^*T(\P^1\times\P^1))\]
                associated to the corresponding moduli space is surjective. Here $\Omega^0(u^*T(\P^1\times\P^1),u^*T(T_{r,0}))$ denotes the sections of $u^*T(\P^1\times\P^1)$ over $D^2$ whose restriction to $\partial D^2$ lies in $u^*T(T_{r,0})$. We assume that suitable Sobolev completions have been chosen.

                Consider the commutative diagram\[\begin{tikzcd}
	0 & 0 \\
	{\Omega^0(u^*T(f^{-1}(z)),u^*T(f^{-1}(z)\cap T_{r,0}))} & {\Omega^{0,1}(u^*T(f^{-1}(z)))} \\
	{\Omega^0(u^*T(\P^1\times\P^1),u^*T(T_{r,0}))} & {\Omega^{0,1}(u^*T(\P^1\times\P^1))} \\
	{\Omega^0(\mathcal{L}_u,\bar{u}^*T(\partial\Delta))} & {\Omega^{0,1}(\mathcal{L}_u)} \\
	0 & 0
	\arrow[from=1-1, to=2-1]
	\arrow[from=1-2, to=2-2]
	\arrow[from=2-1, to=2-2]
	\arrow[from=2-1, to=3-1]
	\arrow[from=2-2, to=3-2]
	\arrow["{\bar{\partial}}", from=3-1, to=3-2]
	\arrow["df", from=3-1, to=4-1]
	\arrow[from=3-2, to=4-2]
	\arrow[from=4-1, to=4-2]
	\arrow[from=4-1, to=5-1]
	\arrow[from=4-2, to=5-2]
\end{tikzcd}\]where the sections of the sheaf $\mathcal{L}_u$ are sections of the pullback bundle $\bar{u}^*T\C=(f\circ u)^*T\C$ that have zeros at $a_1,\cdots,a_m,b_1,\cdots,b_m$. $\mathcal{L}_u$ is isomorphic to $\bar{u}^*T\C(-a_1-\cdots-a_m-b_1-\cdots-b_m)$. $f^{-1}(z)=\{z_1^2z_2=z\}$ is a fiber of $f$ over $z\in\partial\Delta$.

                One can check that the columns are exact sequences. The reason for twisting $\bar{u}^*T\C$ to be $\mathcal{L}_u$ is that the differential $df$ is vanishes at $\{z_1=0\}$ and $\{z_1=\infty\}$, i.e., the double zeros and the double poles of $\bar{u}$. Thus, we twist the target sheaf to make $df$ surjective.

                In order to show the Dolbeault operator in the middle row is surjective, it suffices to show those of the upper row and the lower row are surjective. The upper one is surjective since the bundle pair $(u^*T(f^{-1}(z)),u^*T(f^{-1}(z)\cap T_{r,0}))$ is isomorphic to the trivial pair $(\underline{\C},\underline{\R})$ (see \cite[\S C.1]{J-curve}).

                We now show that the lower one is surjective as well. The degree (relative Chern number) of $(\bar{u}^*T\C,\bar{u}^*T(\partial\Delta))$ is the degree of $\bar{u}$, which is $2m+1$. Twisting once (multiplying all the sections by $z-a_1$ for example) will decrease the degree by $1$. Hence, the twisted bundle pair has degree $1$, and the corresponding Dolbeault operator is automatically surjective.

                \textbf{Step 4.} Disc count.

                In fact, the disc count is the number of values of $\theta$ for which the boundary of $u$ passes through a given point of $T_{r,0}$. The class of the boundary $\partial u$ is $-(2m+1)\partial\beta_0\in H_1(T_{r,0},\Z)$. As $\theta$ varies, the trajectory of each point on $\partial u$ is a circle in the class $\partial\alpha_0\in H_1(T_{r,0},\Z)$.
                Thus, the absolute value of the count is $|-(2m+1)\partial\beta_0\cdot\partial\alpha_0|=2m+1$. 
                
                However, we still need to check that the sign of the count is positive, which amounts to checking that the corresponding moduli space of discs has the desired orientation.
                
                We look back at the commutative diagram in step 3. The orientation of the kernel of the linearized operator $D_{\bpar_J,u}$ is canonically given by that of the Dolbeault operator in the middle row, which is in turn canonically determined by the orientation of the kernels of the upper one and the lower one.

                For the upper row, the bundle pair $(u^*T(f^{-1}(z)),u^*T(f^{-1}(z)\cap T_{r,0}))$ is canonically trivialized by the $S^1$-action on the fibers of $f$, so the kernel is canonically oriented. Moreover, by fixing an orientation of $\partial\Delta$ (and hence a trivialization of $T(\partial\Delta)$), we also obtain a canonical orientation of the kernel of the lower row. (For details on how a trivialization of the totally real subbundle determines an orientation of the moduli space, see \cite[\S 5]{CHC04}.)

                To sum up, the orientation of moduli spaces and hence the sign of the disc count are independent of $m$. Since the count for $m=0$ is positive, the count for general $m$ is positive as well.
            \end{proof}

            In an analogous manner, one can prove the following propositions for other classes. We also state the counterparts of Lemma \ref{bar u} here.
            
            \begin{lemma}\label{counterpart 1}
                There exists a unique (up to an automorphism of the domain) meromorphic function $\bar{u}:(D^2,\partial D^2)\to(\P^1\setminus\Delta,\partial\Delta)$ with exactly $2$ simple poles, $m+1$ double zeros, and $m$ double poles.
            \end{lemma}

            \begin{proposition}
                The torus $T_{r,0}$ bounds a unique $S^1$-family of regular holomorphic discs in the class $(2B+\alpha_0-2\beta_0)+m(A-2\beta_0)$, the disc count of which is $m+1$.
            \end{proposition}

            \begin{lemma}\label{counterpart 2}
                There exists a unique (up to an automorphism of the domain) meromorphic function $\bar{u}:(D^2,\partial D^2)\to(\P^1\setminus\Delta,\partial\Delta)$ with exactly $2$ simple zeros, $m$ double zeros, and $m+1$ double poles.
            \end{lemma}

            \begin{proposition}
                The torus $T_{r,0}$ bounds a unique $S^1$-family of regular holomorphic discs in the class $(A-\alpha_0-2\beta_0)+m(A-2\beta_0)$, the disc count of which is $m+1$.
            \end{proposition}

            The only difference for these two cases is that the meromorphic function $\bar{u}$ given by Lemma \ref{counterpart 1} and Lemma \ref{counterpart 2} can be chosen to be an even function, i.e., $\bar{u}(z)=\bar{u}(-z)$, and replacing the parameter $\theta$ by $\theta+\pi$ yields the same disc up to a reparametrization $z\mapsto -z$. The disc count is hence halved compared to the case of Proposition \ref{disc count}.

            For the classes $\beta_0+m(A-2\beta_0)$, the situation is also slightly different.

            \begin{proposition}
                The torus $T_{r,0}$ bounds a unique $S^1$-family of regular holomorphic discs in the class $\beta_0$ and in the class $A-\beta_0$, the disc count of which is $1$. $T_{r,0}$ bounds no holomorphic discs in the classes $\beta_0+m(A-2\beta_0)$ for $m>1$.
            \end{proposition}

            \begin{proof}
                The result for the class $\beta_0$ is clear, since $\beta_0$ is a basic class. Below, we assume $m\geq1$.
                Suppose $u$ is a holomorphic disc in the class $\beta_0+m(A-2\beta_0)$. 

                Since the class $\beta_0+m(A-2\beta_0)$ has $1$ intersection with $D_\epsilon$, $u$ does not entirely lie over $\P^1\setminus\Delta$. The meromorphic function $\bar{u}=f\circ u:(D^2,\partial D^2)\to(\P^1,\partial\Delta)$ has non-trivial preimage $\bar{u}^{-1}(\Delta)$ of $\Delta$, which is biholomorphic to a disc. 
                
                We consider the complement $C$ of $\bar{u}^{-1}(\Delta)$ in the domain $D^2$. $C$ is a genus $0$ Riemann surface with boundary. (If $\partial(\bar{u}^{-1}(\Delta))$ intersects with $\partial D^2$, it is possible that $C$ is disconnected.) The restriction $\bar{u}|_C:C\to (\P^1\setminus\Delta,\partial\Delta)$ has only double zeros and double poles. If $C$ has a disc component, applying the Riemann-Hurwitz formula to $\bar{u}$ restricted on this component yields a contradiction. Thus, $C$ has no disc component and $\partial(\bar{u}^{-1}(\Delta))$ does not intersect with $\partial D^2$. $C$ has to be an annulus, whose outer boundary $\partial_0C$ is $\partial D^2$, inner boundary $\partial_1 C$ is $\partial(\bar{u}^{-1}(\Delta))$.

                Consider the restriction of $\bar{u}$ on the annulus $C$
                \[\bar{u}|_C:(C,\partial_0 C\cup\partial_1 C)\to(\P^1\setminus\Delta,\partial\Delta).\]
                The further restriction $\bar{u}|_{\partial_0 C}:\partial_0 C\to\partial\Delta$ has degree $2m-1$, since the class of $\partial(\beta_0+m(A-2\beta_0))$ is $(1-2m)\partial\beta_0$. The other restriction $\bar{u}|_{\partial_1 C}:\partial_1 C\to\partial\Delta$ has degree $1$.

                \textbf{Claim.} The restrictions $\bar{u}|_{\partial_0 C}$ and $\bar{u}|_{\partial_1 C}$ should have the same degree.

                In fact, we connect the points $0,\infty\in\P^1\setminus\Delta$ by a smooth curve $\gamma$, and consider the lift $\bar{u}^{-1}(\gamma)$ of the curve $\gamma$. Since $\bar{u}|_C$ has $m$ double zeros and $m$ double poles, $C$ is a branched cover of $\P^1\setminus\Delta$, branching at these zeros and poles. Hence, the lift $\bar{u}^{-1}(\gamma)$ is a union of loops. Since each component of $C\setminus\bar{u}^{-1}(\gamma)$ should have a component of $\partial C$ as a boundary component, $\bar{u}^{-1}(\gamma)$ is actually one loop and generates $H_1(C,\Z)$ (Figure \ref{image_8}). 

                  \begin{figure}[htbp]  
            \centering  
            \includegraphics[width=0.7\textwidth]{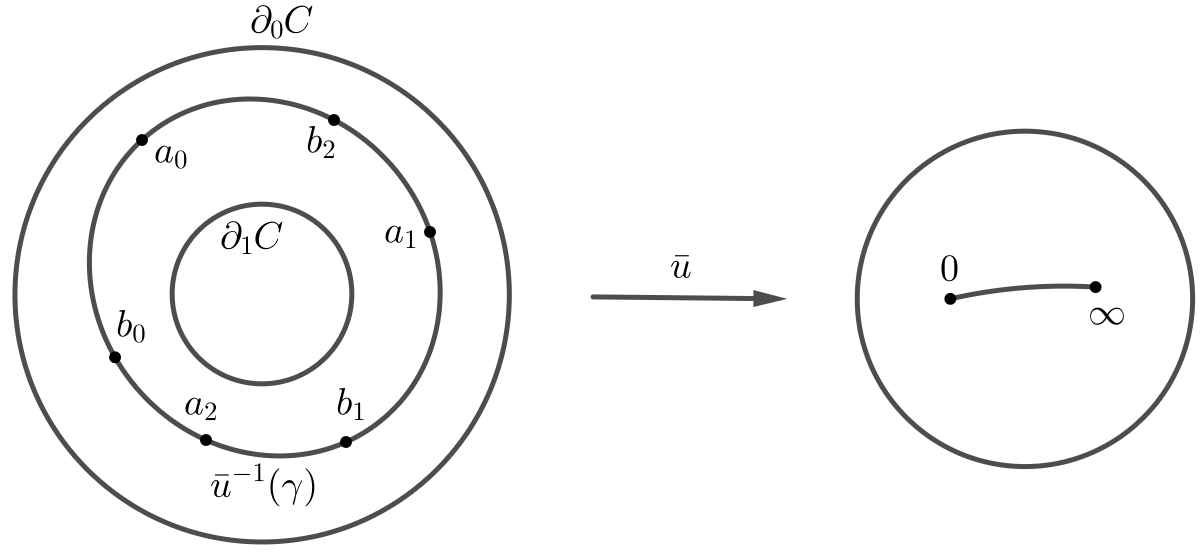}  
            \caption{The restriction of $\bar{u}$ on the annulus $C$}  
            \label{image_8}  
        \end{figure}

                Thus, $\bar{u}^{-1}(\gamma)$ divides the annulus $C$ into two annuli. Restricted on the interior of each of these two annuli, $\bar{u}$ is an honest covering map onto $(\P^1\setminus\Delta)\setminus\gamma$. By considering homology long exact sequences, we can see that the degree of $\bar{u}$ restricted to any of the boundary circles should equal the degree of $\bar{u}$ restricted on any of the divided annuli. The claim is proved.

                It follows from the claim that when $m>1$, such a disc $u$ does not exist. When $m=1$, one can prove that the torus $T_{r,0}$ bounds a unique family of regular discs in $A-\beta_0$, as in Proposition \ref{disc count}, and the corresponding disc count is $1$.
            \end{proof}

            Finally, having understood the contribution of various index $2$ classes to the superpotential, we now find out the expression of the superpotential.

            We choose real coordinates on the moment polytope of $\F_4$ as a subset of $\R^2$, as in Figure \ref{image_6}. In this way, the basic classes $\beta_0$, $B-\beta_0$, $A-\alpha_0-2\beta_0$, and $2B+\alpha_0-2\beta_0$ correspond to the terms in
            \[y+\frac{T^B}{y}+\frac{T^{\frac{A}{2}+B}}{xy^2}+\frac{T^{\frac{A}{2}+B}x}{y^2}.\]
            Here we abused the notation denoting $\omega(A)$ by $A$ and $\omega(B)$ by $B$. Similarly, the class $(A-2\beta_0)$ corresponds to the term $\frac{T^A}{y^2}$.

            \begin{figure}[htbp]  
            \centering  
            \includegraphics[width=0.8\textwidth]{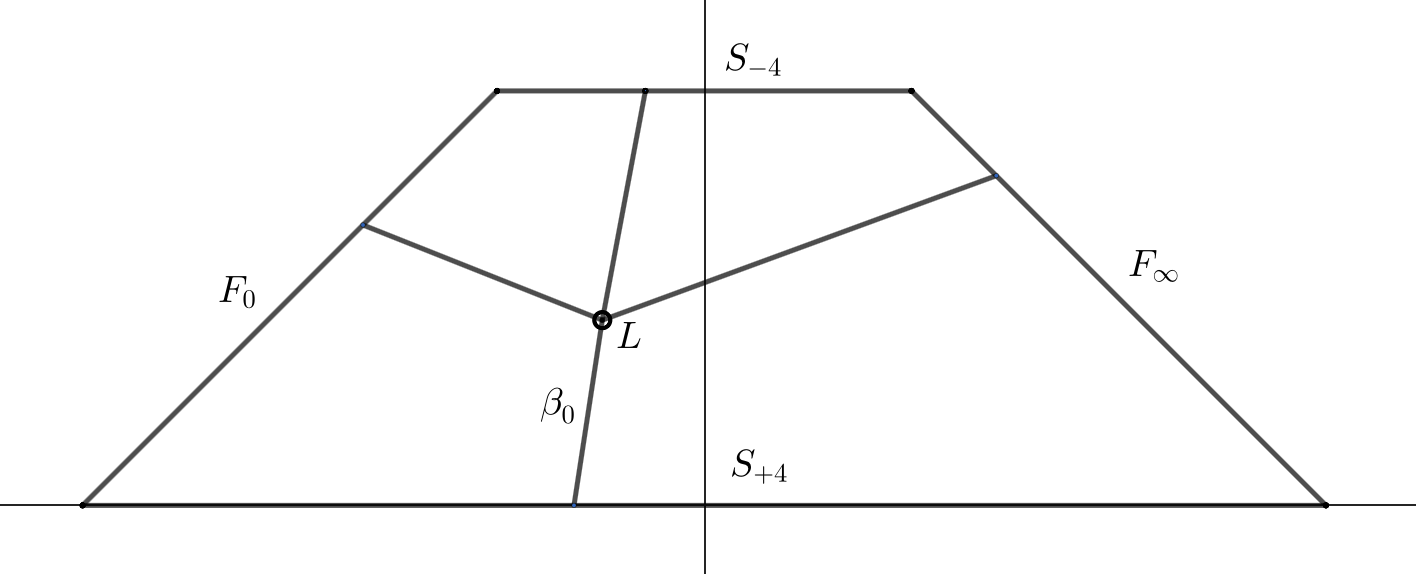}  
            \caption{The moment polytope of $\F_4$ as a subset of $\R^2$}  
            \label{image_6}  
        \end{figure}

            Therefore, by summing up the terms corresponding to various classes listed in Proposition \ref{index 2 discs for F_4}, with appropriate coefficients we just obtained, the expression of the superpotential is
            \begin{align*}
                W(x,y)=y+\frac{T^A}{y}+\frac{T^B}{y}&\left(1+3\frac{T^A}{y^2}+5\frac{T^{2A}}{y^4}+\cdots\right)\\
                +\frac{T^{\frac{A}{2}+B}}{xy^2}&\left(1+2\frac{T^A}{y^2}+3\frac{T^{2A}}{y^4}+\cdots\right)\\
                +\frac{T^{\frac{A}{2}+B}x}{y^2}&\left(1+2\frac{T^A}{y^2}+3\frac{T^{2A}}{y^4}+\cdots\right)
                \end{align*}\begin{equation}\label{superpotential1}
                \ \ \ \ \ \ \ \ \ \ \ =y+\frac{T^A}{y}+T^By\cdot\frac{y^2+T^A}{(y^2-T^A)^2}+T^{\frac{A}{2}+B}y^2\cdot\frac{x+\frac{1}{x}}{(y^2-T^A)^2}.
            \end{equation}

            The critical values of $W$ are $\pm 2T^{\frac{A}{2}}\pm 2T^{\frac{B}{2}}$, which match with those of\[y+\frac{T^B}{y}+x+\frac{T^A}{x},\]
            the superpotential for $\F_0=\P^1\times\P^1$. This is not unexpected since $\F_0$ and $\F_4$ are symplectomorphic.

            \begin{remark}
                The formula of the superpotential given in Theorem \ref{main theorem 2} is related to the formula of $W(x,y)$ here by the change of coordinates $(x,y)\mapsto (T^{-\frac{A}{2}-B}xy^2,y)$, which corresponds to an integral affine change of $\mathbb{R}^2$-coordinates for the moment polytope of $\F_4$.
            \end{remark}

    \section{Comparison with Other Constructions of Mirrors for $\F_4$}

    In Section $3$, we chose a deformation of complex structure of $\F_4$ under which there are no Maslov index $0$ discs bounded by product Lagrangian tori in $\F_4$. In this section, we first consider another deformation, under which index $0$ discs again do not exist, and compute the corresponding superpotential. Then, we explain how the different superpotentials for $\F_4$ we obtained, along with superpotentials for $\F_0$ and $\F_2$, are connected by an infinite sequence of wall-crossing transformations in a scattering diagram. 

    \subsection{Another Mirror Superpotential for $\F_4$}

    Recall that the semiuniversal deformation $\JJ$ of the complex structure of $\F_4$ (see \cite[\S 2.3]{MM04}) is a $\C^3$-family of complex structures. The complex structure over $0$ is that of $\F_4$. There is a two-dimensional algebraic cone in $\C^3$ over which minus $0$ the complex structure is that of $\F_2$. The complex structure over the remaining points in $\C^3$ is that of $\F_0$.

    In this subsection, we take a deformation of $\F_4$ to $\F_2$, and prove that there are no index $0$ discs under this deformation. Then, we determine the corresponding superpotential by showing that under the chosen deformation, product tori in $\F_4$ are related to product tori in $\F_2$ via an explicit wall-crossing transformation. The superpotential (Proposition \ref{Vianna superpotential}), which I learned from Vianna, is also present in \cite{BGL25}.

    For a deformation from $\F_4$ to $\F_2$, the deformed manifold contains a holomorphic sphere with self-intersection $-2$, which is necessarily deformed from the nodal sphere $F_z\cup S_{-4}$ in $\F_4$ for some $z\in\P^1$. In this case, the section $s$ of the obstruction bundle $\OO b$ over $\MM(\F_4;\sigma+2\phi)\cong\P^1\times\P^1$ (Proposition \ref{Ob for F_4}) has the zero set $s^{-1}(0)=\{z_0=z\}\cup\{z_1=z\}$. (The stable spheres in $F_4$ that survive after this deformation are $F_z+F_{z'}+S_{-4}$, $z'\in\P^1$.)

    By applying an automorphism of $\F_4$ induced from an automorphism of $S_{-4}$, we can choose the deformation such that $z=0$. Now we fix a product torus $L$ in $\F_4$, and consider holomorphic discs bounded by $L$.

    \begin{proposition}
        $L$ cannot bound any index $0$ discs after the chosen deformation.
    \end{proposition}

    \begin{proof}
        Suppose $L$ bounds an index $0$ stable disc in the class $m_1\beta_1+m_2\beta_2+m_3\sigma+m_4\phi$, and this stable disc survives after the deformation. 
        
        If this disc becomes a smooth holomorphic disc after the deformation, it should have non-negative intersection with the spheres $F_0,F_\infty,S_{+4},E$, where $E$ is the exceptional sphere in $\F_2$ deformed from $F_0\cup S_{-4}$. We have
            \begin{align*}
                m_1+m_3&\geq0,\\
                m_3&\geq0,\\
                m_4&\geq0,\\
                m_1+m_2-3m_3+m_4&\geq0,\\
                m_1+m_2-2m_3+2m_4&=0.
            \end{align*}
            
            From these relations, the only possible classes are $m\beta_1-m\beta_2$, $m\in\Z_{>0}$. However, stable discs in the class $m\beta_1-m\beta_2$ do not exist in $\F_4$, for the same reason as in Proposition \ref{wall of F_3}.

            If this disc becomes a nodal disc, it has to be the union of a smooth disc of index $0$ and $k(>0)$ copies of $E$ after the deformation. The smooth disc component is in the class 
            \[m_1\beta_1+m_2\beta_2+(m_3-k)\sigma+(m_4-k)\phi,\]which has non-negative intersection with $F_0,F_\infty,S_{+4}$ and positive intersection with $E$. Hence,
            \begin{align*}
                m_3-k&\geq0,\\
                m_4-k&\geq0,\\
                m_1+m_2-3m_3+m_4+2k&=-m_3-m_4+2k>0.
            \end{align*}
            Such a class does not exist. Therefore, $L$ cannot bound any index $0$ discs that survive the deformation.
        \end{proof}

        This proposition shows that there are no walls under the chosen deformation.

        \begin{remark}
            This would not be true if we chose another deformation to $\F_2$ which deforms $F_z\cup S_{-4}$ into the exceptional sphere for some $z\neq0,\infty$, because there exist discs with non-positive index bounded by tori that intersect $F_z$.
        \end{remark}

        Next, we calculate the superpotential by modifying the Lagrangian torus fibration and relating the modified tori to product tori in $\F_2$. 
        
        More specifically, we perform a nodal trade (see for example \cite[\S 8.2]{LagrangianFibrations}) at the top left corner of the moment polytope of $\F_4$.  We require that $F_0\cup S_{-4}$ is deformed into $E$ by the nodal trade. See Figure \ref{image_9}. The second figure is a mutated version of the fibration in the first figure.

        \begin{figure}[htbp] 
            \centering  \includegraphics[width=0.8\textwidth]{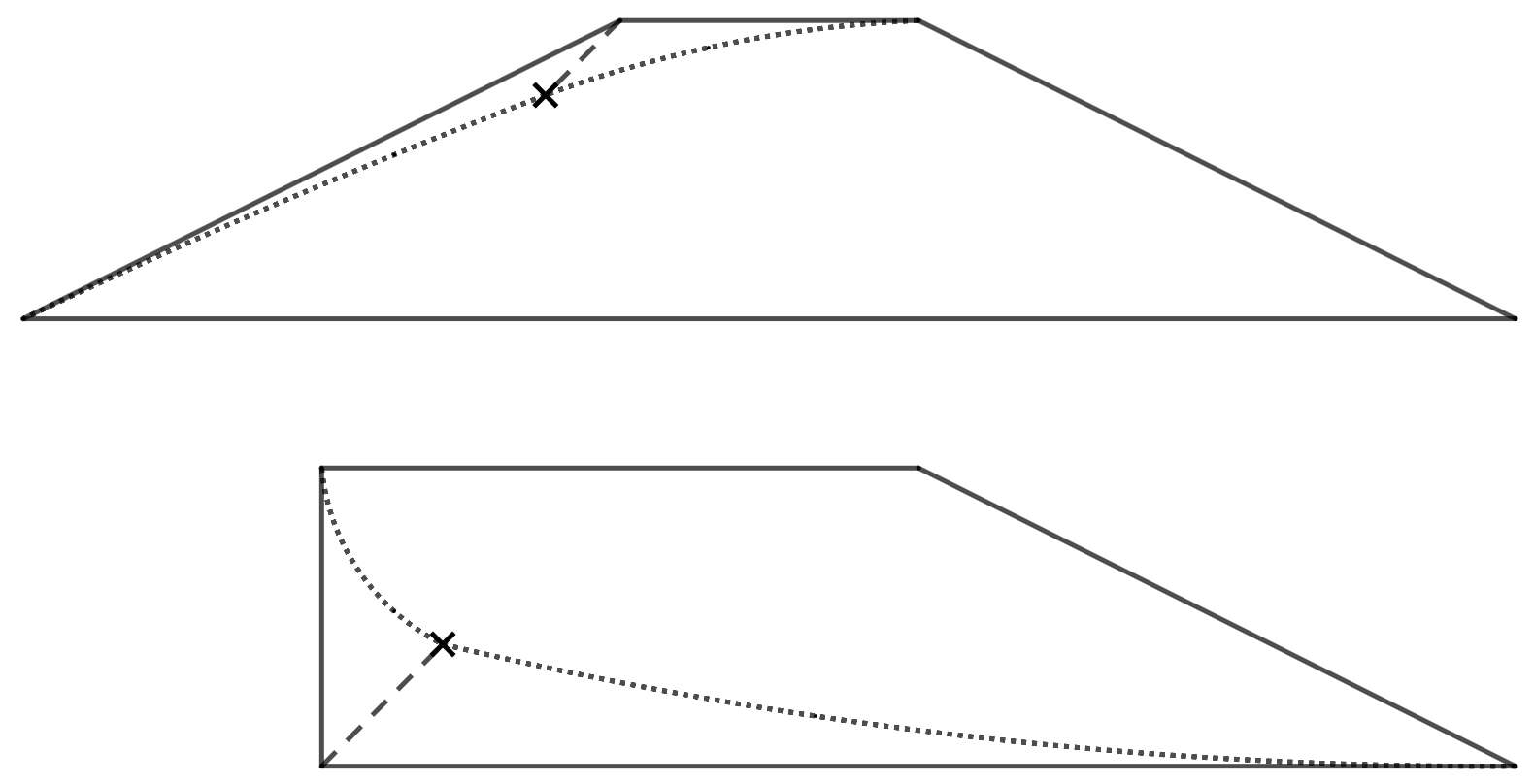} 
            \caption{The base of the Lagrangian torus fibration after the nodal trade}  
            \label{image_9}  
        \end{figure}

        Performing a nodal trade produces a singular Lagrangian fiber, which is represented by an `X' mark in the figures. It also generates a wall, which is roughly indicated by a dashed curve passing through the singular fiber. Note that the fibration represented by the second figure can also be obtained by a nodal trade performed on the standard torus fibration of $\F_2$. In both figures, the lower region delimited by the wall corresponds to product tori in $\F_4$, and the upper region corresponds to product tori in $\F_2$.

        In the following proposition, we use the same affine coordinates for the moment polytope of $\F_4$ as at the end of Subsection $3.3$, which induce coordinates $x,y$ for the mirror space $(\Lambda^*)^2$ as usual.

        \begin{proposition}\label{Vianna superpotential}
            The superpotential for $\F_4$ with respect to the chosen deformation is
            \begin{equation}\label{superpotential2}
                y+\left(1+T^{A-B}\right)\left(\frac{T^B}{y}+\frac{T^{\frac{A}{2}+B}x}{y^2}\right)+\frac{T^{\frac{A}{2}+B}}{xy^2}\left(1+\frac{T^{\frac{A}{2}}x}{y}\right)^3.
            \end{equation}
        \end{proposition}

        \begin{proof}
            Since $\F_2$ is semi-Fano, except for those index $0$ discs originating from the singular fiber, there are no other index $0$ discs. Thus, there is only one wall, i.e., the wall indicated in Figure \ref{image_9}, whose wall-crossing transformation is standard and identical (up to a coordinate change) to that in \cite[\S 5.3]{DA07}. The superpotential for the tori in $\F_4$ is related to the superpotential for $\F_2$ by this wall-crossing transformation.

            We choose affine coordinates for the moment polytope of $\F_2$ in Figure \ref{image_9} that are compatible with those of $\F_4$. They induce coordinates $x',y'$ for the mirror space $(\Lambda^*)^2$ of $\F_2$. The superpotential for $\F_2$ is then given by
            \begin{equation}\label{superpotential F_2}
                y'+\frac{T^B}{y'}+\frac{T^A}{y'}+T^{\frac{A}{2}}x'+\frac{T^{\frac{A}{2}+B}}{x'y'^2}.
            \end{equation}

            According to \cite[\S 5.3]{DA07}, the wall-crossing transformation is given by \begin{equation*}
                \left\{
            \begin{array}{ll}
                \frac{T^B}{y'}=\frac{T^B}{y}+\frac{T^{\frac{A}{2}+B}x}{y^2}, \\
                T^{\frac{A}{2}}x'+y'=y, \\
            \end{array}
            \right. \ \ \ \text{or equivalently,}\ \ \ 
            \left\{
            \begin{array}{ll}
                x'=x\cdot\left(1+\frac{T^{\frac{A}{2}}x}{y}\right)^{-1}, \\
                y'=y\cdot\left(1+\frac{T^{\frac{A}{2}}x}{y}\right)^{-1}. \\
            \end{array}
            \right.
            \end{equation*}
            The proposition then follows by applying this transformation to \eqref{superpotential F_2}.
        \end{proof}

        The marked difference between \eqref{superpotential1} and \eqref{superpotential2} illustrates the sensitive dependence of the SYZ mirror on the choice of perturbation in the non-Fano case. It also suggests that a general choice of regularization of the open Gromov-Witten moduli spaces can yield a mirror construction that differs significantly from those arising naturally via tropical methods (\cite{CPS24}, \cite{BGL25}).

    \subsection{A Scattering Diagram}
            
        So far, we have found two different expressions, \eqref{superpotential1} and \eqref{superpotential2}, for the superpotential, corresponding to two different perturbations of the complex structure, neither of which generates walls.
        It turns out that the expressions \eqref{superpotential1}, \eqref{superpotential2} are related by an infinite sequence of wall-crossing transformations in a scattering diagram. The superpotential \eqref{superpotential F_2} for $\F_2$ and the superpotential for $\F_0$
        \begin{equation}\label{superpotential F_0}
            y+\frac{T^B}{y}+T^{\frac{A}{2}}x+\frac{T^{\frac{A}{2}}}{x}
        \end{equation}
        also appear (correspond to two chambers) in the diagram. Scattering phenomena for this type of examples were first considered in \cite{CPS24}. 

        The main goal of this subsection is to relate the mirror superpotential \eqref{superpotential1} with other expressions stated above; for this reason, we will not give a rigorous definition of scattering diagrams. For a systematic treatment of scattering diagrams, we refer to \cite{GS11_2}, or to \cite{BCHL25} for an account adapted to the symplectic-geometric setting.

        We now construct the scattering diagram. We start with our standard Lagrangian torus fibration of $\F_4$ and perform two nodal trades: one at the top left corner of the moment polytope, the other at the top right corner. See Figure \ref{image_10}.

        \begin{figure}[htbp] 
            \centering  \includegraphics[width=0.8\textwidth]{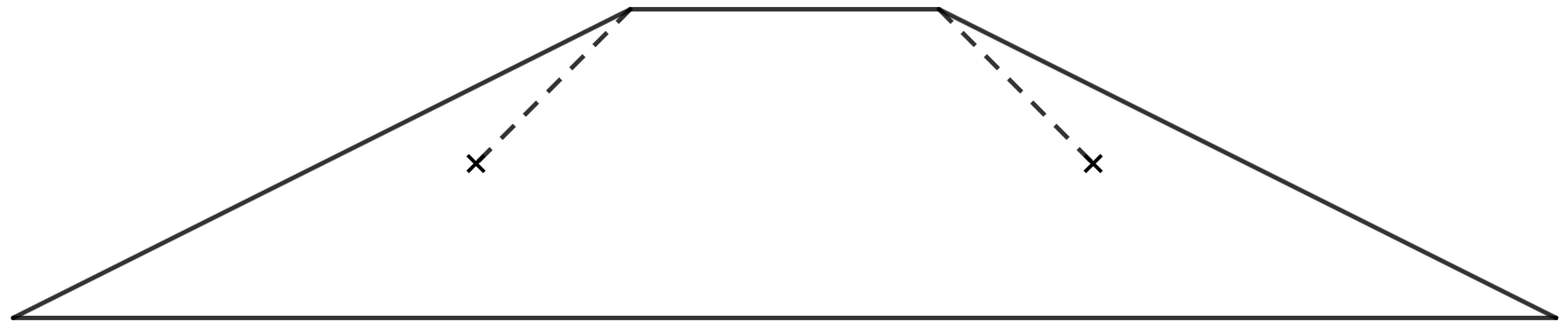} 
            \caption{The base of the Lagrangian torus fibration after two nodal trades}  
            \label{image_10}  
        \end{figure}

        We deform the complex structure accordingly so that $F_0\cup F_\infty\cup S_{-4}$ deforms into a holomorphic sphere. Such a deformation can be constructed by pulling back a deformation from $\F_2$ by the map $\Psi:\F_4\to \F_2$ considered at the beginning of Section $3$, with $\psi$ defined by $z\mapsto z+\frac{1}{z}$.

        Similar to the case in Subsection $4.1$, the resulting fibration can also be obtained by two nodal trades performed on the standard torus fibration of $\F_0$ (at the bottom left corner and at the bottom right corner of the moment polytope of $\F_0$).

        Each of the two nodal trades produces a wall which passes through the singular Lagrangian fiber. These two walls intersect at a point (which stands for a Lagrangian torus). These two walls can scatter additional walls at their intersection (infinitely many in our case, as we will see in Proposition \ref{scattering} below). The tori on these new walls will bound index $0$ discs, whose homotopy classes are positive integer linear combinations of the two classes of the index $0$ discs on the two initial walls. 

        Since $\F_0$ is Fano, apart from the index $0$ discs described above, there are no others. 
        
        This can be better seen in the tropical picture in Figure \ref{image_11}. The two straight lines represent the initial walls, while the rays represent the scattered walls. All of these are images of the Lagrangians that bound index $0$ discs under the map 
        \begin{align*}
            \mathrm{Log}:(\C^*)^2&\to \R^2,\\
            (z_1,z_2)&\mapsto (\mathrm{log}|z_1|,\mathrm{log}|z_2|),
        \end{align*}
        where the domain $(\C^*)^2$ is the complement of the toric boundary in $\F_0$. 

        \begin{figure}[htbp] 
            \centering  \includegraphics[width=0.3\textwidth]{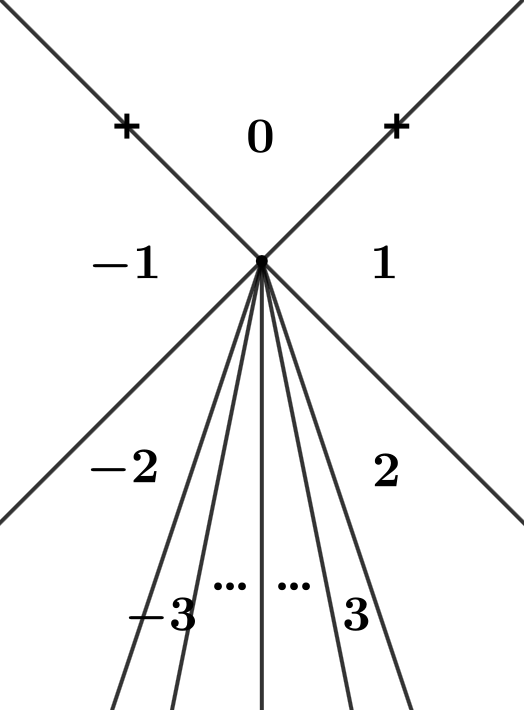} 
            \caption{The scattering diagram}  
            \label{image_11}  
        \end{figure}

        As usual, each chamber divided by these walls corresponds to a family of Lagrangian tori, whose affine coordinates induce coordinates on the uncorrected mirror space $(\Lambda^*)^2$. 
        
        Since this diagram can also be obtained by performing a nodal trade at each of the two bottom corners of the moment polytope of $\F_0$, the topmost chamber corresponds to product tori in $\F_0$. We label this chamber as $0$ and denote the induced coordinates as $(x_0,y_0)\in(\Lambda^*)^2$.

        Starting from the topmost chamber, after crossing one of the two initial walls, we reach one of the two adjacent chambers to the left and right of the topmost chamber. The tori in these two chambers are both product tori in \( \F_2 \). We label these chambers as \(-1\) and \(1\), with the induced coordinates denoted as \((x_{-1}, y_{-1})\) and \((x_1, y_1)\), respectively.

        The initial wall-crossing transformations are clear, namely,
        \begin{equation*}
                \left\{
            \begin{array}{ll}
                x_1=x_0\cdot\left(1+\frac{T^\frac{A}{2}}{x_0y_0}\right)^{-1}, \\
                y_1=y_0\cdot\left(1+\frac{T^\frac{A}{2}}{x_0y_0}\right), \\
            \end{array}
            \right. \ \ \ \text{and}\ \ \ 
            \left\{
            \begin{array}{ll}
                x_{-1}=x_{0}\cdot\left(1+\frac{T^\frac{A}{2}x_{0}}{y_{0}}\right), \\
                y_{-1}=y_{0}\cdot\left(1+\frac{T^\frac{A}{2}x_{0}}{y_{0}}\right). \\
            \end{array}
            \right.
        \end{equation*}

         For the initial wall that separates Chamber $0$ and Chamber $1$, we can read off the homotopy class of the index $0$ discs on the wall from the wall-crossing factor $\left(1+\frac{T^\frac{A}{2}x_{0}}{y_{0}}\right)$, thereby confirming that the slope of this wall in the tropical picture (as a line in $\R^2$) is 1. Similarly, the slope of the initial wall that separates Chamber $0$ and Chamber $-1$ is $-1$.

        Starting from the relative positions of the two initial walls and their wall-crossing transformations, one can determine the structure of the entire scattering diagram.

        \begin{proposition}\label{scattering}
            There are infinitely many scattered walls, and the divided chambers can be labeled with all integers in a clockwise manner, as in Figure \ref{image_11}. Denote the coordinates induced by the tori in Chamber $k$ as $(x_k,y_k)$. The wall separating Chamber $k$ and Chamber $k+1$ for $k\geq 1$ has slope $-2k+1$, and the associated wall-crossing transformation is\begin{equation*}
                \left\{
            \begin{array}{ll}
                x_{k+1}=x_k\cdot\left(1+T^{\frac{2k-1}{2}A}x_ky_k^{-2k+1}\right)^{2k-1}, \\
                y_{k+1}=y_k\cdot\left(1+T^{\frac{2k-1}{2}A}x_ky_k^{-2k+1}\right). \\
            \end{array}
            \right. 
        \end{equation*}The wall separating Chamber $-k$ and Chamber $-k-1$ for $k\geq 1$ has slope $2k-1$, and the associated wall-crossing transformation is\begin{equation*}
                \left\{
            \begin{array}{ll}
                x_{-k-1}=x_{-k}\cdot\left(1+T^{\frac{2k-1}{2}A}x_{-k}^{-1}y_{-k}^{-2k+1}\right)^{-2k+1}, \\
                y_{-k-1}=y_{-k}\cdot\left(1+T^{\frac{2k-1}{2}A}x_{-k}^{-1}y_{-k}^{-2k+1}\right). \\
            \end{array}
            \right. 
        \end{equation*}
            Apart from the above, there remains one vertical wall, represented by the middle ray in Figure \ref{image_11}, whose wall-crossing transformation is\begin{equation*}
                \left\{
            \begin{array}{ll}
                x_{+\infty}=x_{-\infty}\cdot\left(1-\frac{T^A}{y^2}\right)^{4}, \\
                y_{+\infty}=y_{-\infty}. \\
            \end{array}
            \right. 
        \end{equation*}
        \end{proposition}

        \begin{proof}
            Given the initial walls and their wall-crossing transformations, there is a systematic and essentially unique procedure to insert rays to achieve consistency for the scattering diagram (see \cite{KS06} and \cite[\S 4.3]{GS11}). 
            
            In our case, our initial conditions match those of the scattering diagram in Example $6.41$ in \cite[\S 6.3]{Tropical} (the case where $l_1=l_2=2$), up to a linear change of logarithmic coordinates in $\R^2$. Thus, our resulting scattering diagram matches the one in Example $6.41$ in \cite{Tropical} up to the change of coordinates. The proposition then follows.

            Here are some remarks about the vertical wall. The coordinates\[(x_{+\infty},y_{+\infty}):=\lim_{k\to+\infty}(x_k,y_k)\in(\Lambda^*)^2\]are well-defined over the Novikov field $\Lambda$, because according to the wall-crossing transformation from $(x_k,y_k)$ to $(x_{k+1},y_{k+1})$, the sequence $\{(x_k,y_k)\}$ eventually stabilizes modulo any large power $T^C$ of the Novikov parameter. Similarly, the coordinates $(x_{-\infty}, y_{-\infty})$ are well-defined. $(x_{-\infty}, y_{-\infty})$ and $(x_{+\infty},y_{+\infty})$ are related by the last transformation in the proposition.
        \end{proof}

        Finally, we are able to relate the superpotentials we obtained using the scattering diagram.

        The superpotential for the tori in Chamber $0$ is that for $\F_0$, namely
        \[y_0+\frac{T^B}{y_0}+T^{\frac{A}{2}}x_0+\frac{T^{\frac{A}{2}}}{x_0}.\]
        The superpotentials for Chamber $-1$ and Chamber $1$ are those for $\F_2$. The one for Chamber $1$ is \eqref{superpotential F_2}, namely
        \[y_1+\frac{T^B}{y_1}+\frac{T^A}{y_1}+T^{\frac{A}{2}}x_1+\frac{T^{\frac{A}{2}+B}}{x_1y_1^2}.\] Similarly, the one for Chamber $-1$ is 
        \[y_{-1}+\frac{T^B}{y_{-1}}+\frac{T^A}{y_{-1}}+\frac{T^{\frac{A}{2}}}{x_{-1}}+\frac{T^{\frac{A}{2}+B}x_{-1}}{y_{-1}^2}.\]
        The superpotential for Chamber $2$ is the one $\eqref{superpotential2}$ in Proposition \ref{Vianna superpotential}, namely \[y_2+\left(1+T^{A-B}\right)\left(\frac{T^B}{y_2}+\frac{T^{\frac{A}{2}+B}x_2}{y_2^2}\right)+\frac{T^{\frac{A}{2}+B}}{x_2y_2^2}\left(1+\frac{T^{\frac{A}{2}}x_2}{y_2}\right)^3,\]which is already a superpotential for $\F_4$. Similarly, the one for Chamber $-2$ is
        \[y_{-2}+\left(1+T^{A-B}\right)\left(\frac{T^B}{y_{-2}}+\frac{T^{\frac{A}{2}+B}}{x_{-2}y_{-2}^2}\right)+\frac{T^{\frac{A}{2}+B}x_{-2}}{y_{-2}^2}\left(1+\frac{T^{\frac{A}{2}}}{x_{-2}y_{-2}}\right)^3.\]
        
        By successively applying the wall-crossing transformations from Proposition \ref{scattering}, we obtain infinitely many candidates for the superpotential of $\F_4$. Denote by $W_{k}(x_k,y_k)$ the superpotential for Chamber $k$. There is a well-defined limit
        \[W_{+\infty}(x,y):=\lim_{k\to +\infty}W_k(x,y)\in\Lambda((x,x^{-1},y,y^{-1})),\]since the sequence $\{W_k(x,y)\}$ eventually stabilizes modulo any $T^C$. It turns out that 
        \[W_{+\infty}(x,y)=y+\frac{T^A}{y}+\frac{T^B}{y}\left(\sum_{k=0}^\infty (2k+1)\left(\frac{T^A}{y^2}\right)^k\right)+\frac{T^{\frac{A}{2}+B}}{xy^2}
                +\frac{T^{\frac{A}{2}+B}x}{y^2}\left(\sum_{k=0}^\infty \binom{k+3}{3}\left(\frac{T^A}{y^2}\right)^k\right),\]
        \[=y+\frac{T^A}{y}+T^By\cdot\frac{y^2+T^A}{(y^2-T^A)^2}+\frac{T^{\frac{A}{2}+B}}{xy^2}+\frac{T^{\frac{A}{2}+B}x}{y^2}\cdot\left(1-\frac{T^A}{y^2}\right)^{-4}\]
        \[=W\left(x\cdot\left(1-\frac{T^A}{y^2}\right)^{-2},y\right),\]where $W$ denotes the superpotential \eqref{superpotential1} we obtained in Section $3$, namely
        \[W(x,y)=y+\frac{T^A}{y}+T^By\cdot\frac{y^2+T^A}{(y^2-T^A)^2}+\left(\frac{T^{\frac{A}{2}+B}}{xy^2}
                +\frac{T^{\frac{A}{2}+B}x}{y^2}\right)\cdot\left(1-\frac{T^A}{y^2}\right)^{-2}.\]
        
        Similarly, we have
        \[W_{-\infty}(x,y)=W\left(x\cdot\left(1-\frac{T^A}{y^2}\right)^2,y\right).\]

        The above relationships among \( W_{-\infty} \), \( W \), and \( W_{+\infty} \) indicate that the vertical wall in Figure \ref{image_11} should be regarded as two coinciding walls, whose wall-crossing transformations are both $x\mapsto x\left(1-\frac{T^A}{y^2}\right)^2$. Moreover, the superpotential \( W(x,y) \) corresponds to the `chamber between these two walls' (which, under the deformation chosen in this subsection, does not exist).

    \bibliographystyle{plain}
    \bibliography{ref}

@article{MM04,
	author = "M. Manetti",
	title = "Lectures on deformations of complex manifolds",
        year = "2004",
        journal = "Rendiconti di Matematica, Serie VII",
        volume = "24",
        pages = "1-183"
}

@article{DA07,
	author = "D. Auroux",
	title  = "{Mirror symmetry and T-duality in the complement of an anticanonical divisor}",
	year = "2007",
        journal = "Journal of Gökova Geometry Topology",
        volume = "1",
        pages = "51-91"
}

@article{DA09,
	author = "D. Auroux",
	title  = "{Special Lagrangian fibrations, wall-crossing, and mirror symmetry}",
	year = "2009",
        journal = "Surveys in Differential Geometry",
        volume = "13",
        pages = "1-47"
}

@article{AAK16,
	author = "M. Abouzaid and D. Auroux and L. Katzarkov",
	title  = "{Lagrangian fibrations on blowups of toric varieties and mirror symmetry for hypersurfaces}",
	year = "2016",
        journal = "Publications mathématiques de l'IHÉS",
        volume = "123",
        pages = "199-282"
}

@article{DA23,
	author = "D. Auroux",
	title  = "{Holomorphic discs of negative Maslov index and extended deformations in mirror symmetry}",
	year = "2023",
        journal = "arXiv preprint arXiv:2309.13010",
}

@article{HY25,
      title="{Family Floer program and non-archimedean SYZ mirror construction}", 
      author="H. Yuan",
      year="2025",
      journal = "arXiv preprint arXiv:2003.06106"
}

@article{BGL25,
      title="{Gromov-Witten invariants and mirror symmetry for non-Fano varieties via tropical disks}", 
      author="P. Berglund and T. Gräfnitz and M. Lathwood",
      year="2025",
      journal = "arXiv preprint arXiv:2404.16782"
}

@article{CPS24,
	author = "M. Carl and M. Pumperla and B. Siebert",
	title = "{A tropical view on Landau-Ginzburg models}",
        year = "2024",
        journal = "Acta Mathematica Sinica",
        volume = "40",
        pages = "329-382"
}

@article{CL14,
	author = "K. Chan and S.-C. Lau",
	title = "{Open Gromov-Witten invariants and superpotentials for semi-Fano toric surfaces}",
        year = "2014",
        journal = "International Mathematics Research Notices",
        volume = "2014(14)",
        pages = "3759-3789"
}

@article{FOOO10,
	author = "K. Fukaya and Y.-G. Oh and H. Ohta and K. Ono",
	title = "{Lagrangian Floer theory on compact toric manifolds, I}",
        year = "2010",
        journal = "Duke Mathematical Journal",
        volume = "151(1)",
        pages = "23-175"
}

@article{FOOO12,
	author = "K. Fukaya and Y.-G. Oh and H. Ohta and K. Ono",
	title = "{Toric degeneration and non-displaceable Lagrangian tori in $S^2\times S^2$}",
        year = "2012",
        journal = "International Mathematics Research Notices",
        volume = "2012(13)",
        pages = "2942–2993"
}

@article{SYZ96,
	author = "A. Strominger and S.-T. Yau and E. Zaslow",
	title = "{Mirror symmetry is T-duality}",
        year = "1996",
        journal = "Nuclear Physics B",
        volume = "479",
        pages = "243-259"
}

@article{CHC04,
	author = "C.-H. Cho",
	title = "{Holomorphic disc, spin structures and Floer cohomology of the Clifford torus}",
        year = "2004",
        journal = "International Mathematics Research Notices",
        volume = "2004(135)",
        pages = "1803-1843"
}

@article{HV00,
    author = "K. Hori and C. Vafa",
    title = "{Mirror symmetry}",
    journal = "arXiv preprint hep-th/0002222",
    year = "2000"
}

@article{CO06,
    author = "C.-H. Cho and Y.-G. Oh",
    title = "{Floer cohomology and disc instantons of Lagrangian torus fibers in Fano toric manifolds}",
    journal = "Asian Journal of Mathematics",
    volume = "10",
    pages = "773-814",
    year = "2006"
}

@book{AlgebraicSurfaces,
series={London Mathematical Society Student Texts}, title={Complex Algebraic Surfaces}, 
publisher={Cambridge University Press}, 
author={A. Beauville}, 
year={1996}}

@book{J-curve,
    author = "D. McDuff and D. Salamon",
    title = "J-holomorphic Curves and Symplectic Topology",
    publisher = "American Mathematical Society",
    year = "2012"
}

@book{LagrangianFibrations,
    author = "J. Evans",
    title = "Lectures on Lagrangian Torus Fibrations",
    publisher = "Cambridge University Press",
    year = "2023"
}

@book{Tropical,
    author = "M. Gross",
    title = "Tropical Geometry and Mirror Symmetry",
    publisher = "American Mathematical Society",
    year = "2011"
}

@article{KS06,
    author = "M. Kontsevich and Y. Soibelman",
    title = "{Affine structures and non-Archimedean analytic spaces}",
    journal = "The Unity of Mathematics, Progress in Mathematics",
    volume = "244",
    pages = "321-385",
    year = "2006"
}

@article{GS11,
    author = "M. Gross and B. Siebert",
    title = "{An invitation to toric degenerations}",
    journal = "Surveys in Differential Geometry",
    volume = "16(1)",
    pages = "43-78",
    year = "2011"
}

@article{Chan11,
    author = "K. Chan",
    title = "{A formula equating open and closed Gromov-Witten invariants and its applications to mirror symmetry}",
    journal = "Pacific Journal of Mathematics",
    volume = "254(2)",
    pages = "275-293",
    year = "2011"
}

@article{GS11_2,
    author = "M. Gross and B. Siebert",
    title = "{From real affine geometry to complex geometry}",
    journal = "Annals of Mathematics",
    volume = "174(3)",
    pages = "1301-1428",
    year = "2011"
}

@article{BCHL25,
    author = "S. Bardwell-Evans and M.-W. M. Cheung and H. Hong and Y.-S. Lin",
    title = "{Scattering diagrams from holomorphic discs in log Calabi-Yau surfaces}",
    journal = "Journal of Differential Geometry",
    volume = "130(1)",
    pages = "71-150",
    year = "2025"
}

@book{FOOObook,
    author = "K. Fukaya and Y.-G. Oh and H. Ohta and K. Ono",
    title = "{Lagrangian Intersection Floer Theory: Anomaly and Obstruction, Part I}",
    series = "AMS/IP Studies in Advanced Mathematics",
    publisher = "American Mathematical Society",
    year = "2009"
}

@article{GHK15,
    author = "M. Gross and P. Hacking and S. Keel",
    title = "{Mirror symmetry for log Calabi-Yau surfaces I}",
    journal = "Publications Mathématiques de l'IHÉS",
    volume = "122",
    pages = "65-168",
    year = "2015"
}

@article{GS22,
    author = "M. Gross and B. Siebert",
    title = "{The canonical wall structure and intrinsic mirror symmetry}",
    journal = "Inventiones Mathematicae",
    volume = "229",
    pages = "1101-1202",
    year = "2022"
}

@article{GS19,
	author = "M. Gross and B. Siebert",
	title  = "{Intrinsic mirror symmetry}",
	year = "2019",
        journal = "arXiv preprint arXiv:1909.07649",
}

\end{document}